\documentclass[reqno,10pt]{amsart}
\usepackage{color} 
\usepackage[left=1 in, right=1 in,top=1 in, bottom=1 in]{geometry}
\usepackage{amsfonts,amsmath,amsthm,amssymb,mathrsfs,environ}
\newtheorem{theorem}{Theorem}[section]

\newtheorem{lemma}[theorem]{Lemma}

\newtheorem{remark}[theorem]{Remark}

\usepackage[pdfstartview=FitH,colorlinks=true,backref=page]{hyperref}
\hypersetup{urlcolor=red, linkcolor=blue,citecolor=red, CJKbookmarks=true}
\allowdisplaybreaks
\makeatletter

\newcommand{\bd}{\mathrm{d}}

\newcommand{\vdot}{v\!\cdot\!\nabla_x}
\newcommand{\sdot}{\!\cdot\!}
\newcommand{\m}{\dm}
\newcommand{\bdv}{\dm \bd v}
\newcommand{\dm}{\mathrm{M}}
\newcommand{\divg}{\mathrm{div}}
\newcommand{\Q}{\mathcal{Q}}
\newcommand{\timess}{\!\times\!}
\newcommand{\curl}{\nabla\!\!\times\!\!}

\newcommand{\reps}{\tfrac{1}{\epsilon}}
\newcommand{\seps}{\sqrt{\epsilon}}

\newcommand{\repst}{\tfrac{1}{\epsilon^2}}

\newcommand{\llv}{L^2_{v,\Lambda}}

\newcommand{\dt}{\frac{\mathrm{d}}{2 \mathrm{d}t}}
\newcommand*{\rom}[1]{\expandafter\@slowromancap\romannumeral #1@}

\newcommand{\bdp}{\mathcal{P}}

\newcommand{\intps}{\int_{\scriptstyle \mathbb{R}^3}\!\int_{\scriptstyle \mathbb{R}^3} }

\newcommand{\intv}{\int_{\mathbb{R}^3}}
\newcommand{\intt}{\int_{\mathbb{R}^3}}

\newcommand{\tdt}{\frac{\mathrm{d}}{ \mathrm{d}t}}
\newcommand{\lrf}{\tfrac{ \epsilon E_\epsilon +  v\times B_\epsilon}{\epsilon \m}}
\newcommand{\hd}{\mathcal{D}^s_\epsilon(t) \sqrt{\mathcal{H}_\epsilon^s}(t)}
\makeatother
\NewEnviron{aligno}{%
\begin{align}\begin{aligned}
  \BODY
\end{aligned}\end{align}
}

\title[ VMB]{ From two species Vlasov-Maxwell-Boltzmann system to magnetohydrodynamics system }

\author{Xu Zhang}
\address[X. Zhang]{\newline School of Mathematics and Statistics, zhengzhou University, Zhengzhou, 450001, P. R. China}
\email{xuzhang889@zzu.edu.cn}

\begin{document}
	\begin{abstract}
  In this work, the magnetohydrodynamics system is formally derived from two species Vlasov-Maxwell-Boltzmann system. By employing the hypocoercivity of the linear Boltzmann  operator and overcoming the difficulties resulting from the singular Lorentz term, we first obtain the  uniform estimates of solutions with respect to the Knudsen number and then derive the magnetohydrodynamics system  from the dimensionless Valsov-Maxwell-Boltzmann system.

	\end{abstract}
\maketitle

\section{Introduction and Motivation}
The plasma where the charged particles in dilute gas  move under the influence of the  self-consistent electromagnetic field and collisions can be described by the Vlasov-Maxwell-Boltzmann (VMB) system. For  two species case with equal mass and charge (e.g. anion and cation), the VMB system on the torus in the three dimension is
\begin{align}
\label{vmb}
\begin{cases}
\partial_t   f^+ + \vdot f^+  +  (   E  +    v\timess B ) \sdot\nabla_v f^+  =    \Q(f^+ ,f^+ ) +   \Q(f^+ ,f^- ), \\
 \partial_t   f^-  + \vdot f^-  - (    E  +    v\timess B ) \sdot\nabla_v f^-  =    \Q(f^- ,f^+ ) +  \Q(f^- ,f^- ), \\
 \varepsilon_0 \mu_0\partial_t   E  - \curl   B  = - \mu_0 \intv (f^+ - f^-) v \bd v,\\
\partial_t   B  + \curl   E  =0,\\
\divg   B  =0,~~ \divg   E  = \tfrac{1}{\epsilon_0}\intv (f^+  - f^- ) \bdv.
\end{cases}
\end{align}
In the above system, $f^\pm=f^\pm(t,x,v)(x \in \mathbb{T}^3, ~v \in \mathbb{R}^3)$   are the number density of anion or cation respectively. $B$ and $E$ are electric and magnetic field. The $\Q(f^\pm, f^\pm)$ describe the collision between $f^\pm$ and $f^\pm$ and will be detailed later. In \eqref{vmb}, the first two equations  are kinetic equations and describe that the anions or cations move along their trajectories under the Lorentz force ($ E +    v\timess B) \sdot\nabla_v f^\pm$) and collisions.  The third and forth equations are Maxwell's system which  models how electromagnetic field are self-generated. Specially,  $\varepsilon_0$  and $\mu_0$ are the electric permittivity and magnetic permeability respectively. Furthermore, $\varepsilon_0$  and $\mu_0$ satisfy
\[ c^2 \varepsilon_0 \mu_0 =1,~~c~\text{is the speed of light.}  \]
The couplling of kinetic  system and Maxwell system shows that the moving charged particles in electro-magetic field change the status of the current which in turn affect the creation of the field.

The collission operator $\Q$
\begin{equation}\label{collision-original}
\begin{split}
\Q(f,g)=&\int_{\mathbb{R}^3\times\mathbb{S}^{2}}(f'g'_*-fg_*) b(v-v_*,\omega) \bd v_*\bd\omega\,,
\end{split}
\end{equation}
with $f'=f(v')$, $g'_*= g(v'_*)$, $g_*= g(v_*)$. Here, $v$ and $v_*$ are the velocities of two particles before the   elastic collision and $v'$ and $v_*'$ denotes  velocities after collision. Moreover, the velocities of particles satisfy
\begin{align*}
\begin{cases}
v+ v_* = v' + v'_*, \\
|v|^2 + |v_*|^2 =    |v'|^2 + |v'_*|^2\, ,
\end{cases}
\end{align*}
and
\begin{align*}
\begin{cases}
v' = \frac{v + v_*}{2} + \frac{|v-v_*|}{2}\omega \\
v_*'= \frac{v + v_*}{2} - \frac{|v-v_*|}{2}\omega\,, ~~\omega   \in \mathbb{S}^2.
\end{cases}
\end{align*}

In this work, the cross section $b$  is assumed to be hard potential, i.e, the exists constant $C_\Phi$ such that
\[ b(v-v_*,\omega)= C_\Phi |v-v_*|m(\cos\hat{\theta}),~~\cos\hat\theta = \langle \tfrac{v-v_*}{|v-v_*}, \omega\rangle. \]
Moreover, there exists some positive constant $C_b$ such that
\[ \forall z \in [-1,1],~~ |m(z)| \le C_b, ~~ |m'(z)| \le C_b. \]
Furthermore, $\Q(f,f)$ enjoys the following properties:
\[ \intv \Q(f,f) \bd v =0,~~\intv\Q(f,f)v\bd v =0,~~\intv \Q(f,f) \tfrac{|v|^2-3}{2} \bd v =0.   \]

From the point view of physics, the small parameters (such as Debye number and  Knudsen number) play a key role in describing the transition phenomenon. The Knudsen number which is defined as the ratio of the mean free path  length of the molecular to  representative physical length scale determines whether model of statistical mechanics or continuum mechanics is used.   From the point view of physics, while the { Knudsen number} is very small, the dilute gas under consideration are  fluid regimes. In this work, we only consider the incompressible fluid regimes where the dynamics behavior of the charged particles are described by Navier-Stokes type equations. According to different kinds of scalings of the electromagnetic field in the VMB system, the electric and magnetic field may vanish or preserve in the limiting fluid system.
The limiting fluid system could be Navier-Stokes (NS) euations (the affect of electromagnetic is negligible), Navier-Stokes-Poisson (NSP) system (the affect of the magetic field is negligible), resistive magnetohydrodynamics (MHD) system (the affect of the electric field is negligible) and Navier-Stokes-Maxwell (NSW) system (both fields are important).  The formal derivation of NS, NSP and NSW system from VMB system can be found in \cite{diogosrm-2019-vmb-fluid}. The dimensionless VMB system to MHD system  can be found in \cite{jama2012siam}.

In this work, we concern with the rigorous justification work.  Based on our knowledge,  the Navier-Stokes-Maxwell, Navier-Stokes-Poisson and Navier-Stokes` limits of VMB system have been rigorously verified in the current literature (see \cite{vmbtonswh,vmbtonswu,mvmbtothree,vmbtonsp} for instance). The Navier-Stokes-Poisson limit  of two species VMB system was verified in \cite{vmbtonsp} where the magnetic field in VMB system disappears in the limiting system. The Navier-Stokes-Maxwell  limit   can be found in \cite{vmbtonswh,vmbtonswu} where both electric field and magnetic field preserve in the limiting system. The resistive magnetohydrodynamics system is also a very important model in describing the status of plasma and has not been rigorously derived from VMB system. Besides, there exist new technical difficulties during the justification.  The concern of this work is to justify the transition of system \eqref{vmbtomhd} to the following  MHD system:
\begin{align}
\label{mhddd}
\begin{cases}
\partial_t u + u\sdot \nabla u - \nu \Delta u + \nabla P = (\curl B) \timess B,\\
\partial_t \theta + u \sdot\nabla \theta - \kappa \Delta \theta =0,\\
\divg u = \divg B= 0,~~ \rho + \theta =0,\\
\partial_t B  - \tfrac{1}{\sigma} \Delta B -\curl(u \times B) =0.
\end{cases}
\end{align}
In \eqref{mhddd}, $u$ is the velocity of fluid. $\rho$ and $\theta$ are density and temperature respectively. $P$ is the pressure.
The positive constants $\nu,~~\kappa$ and $\sigma$ have an explanation from the point view of  the kinetic system (see \eqref{nusigma}).   Compared to the Maxwell system in \eqref{vmbtomhd} where there exist both electric and magnetic field, only magnetic field preserves as the Knudsen go to zero.

The formal derivation of MHD equations  from VMB system was  performed in \cite{jama2012siam}. But the Lorentz force term in their dimensionless system is extremely singular and hard to be controlled. In what follows,  we try to derive a new dimensionless VMB system to overcome this difficulty. To deduce the MHD system from the VMB system, we need to perform the dimensionless analysis to system \eqref{vmb}.  The dimensionless Boltzmann equation (see \cite{saint2009book} and \cite{diogosrm-2019-vmb-fluid} for instance) is
\[  \mathrm{St} \partial_t f + \vdot f = \tfrac{1}{\mathrm{Kn}}\Q(f,f), \]
where $\mathrm{St}$ is the { Strouhal number} and $\mathrm{Kn}$ is the { Knudsen number}.
In the incompressible regimes,
\[ \mathrm{St} = \mathrm{Kn} = \epsilon. \]
For the dimensionless analysis, we first introduce the new variables $\tilde{t},~\tilde{x},~\tilde{v}$ with
\[ t =  \tfrac{\tilde{t}}{\epsilon},~ x =  \tfrac{\tilde{x}}{\epsilon},~v =  \tfrac{\tilde{v}}{\epsilon}. \]
Inspired by \cite{jama2012siam}, to get the MHD system, the electric permittivity $\epsilon_0$ is assumed to be very small and magnetic permeability $\mu_0$ be $O(1)$ (in fact, $\mu_0$ is a small constant, withou loss of generality, we set it be one), i.e.,
\[  \epsilon_0 = \epsilon,~~\mu_0 = 1. \]
Under the above setting, the speed of light tends to infinity as $\epsilon$ goes to zero. Since we are performing the dimensionless analysis, this unrealistic discrepancy  should be understood  as asymptotic regimes where proper physical approximations are still valid.

Then, the dimensionless functions for  \eqref{vmb} are defined as follows:
\begin{align*}
{f}^\pm(t,x,v)= \epsilon^3 {f}^\pm_\epsilon(\tilde{t}, \tilde{x}, \tilde{v}), ~~E(t,x) = {E}_\epsilon(\tilde{t},\tilde{x}),~~B(t,x) = {B}_\epsilon(\tilde{t},\tilde{x}),\\
\tilde{\Q}(\tilde{f}, \tilde{f}) = \int_{\mathbb{R}^3\times\mathbb{S}^{2}}(\tilde{f}' \tilde{f}'_*- \tilde{f}\tilde{f}_*) \tilde{b}(\tilde{v}-\tilde{v}_*,\omega) \bd \tilde{v}_*\bd\omega,~~ ~~{b}(v,\omega)=\reps {\tilde{b}}(\tilde{v},\omega).
\end{align*}
Plugging the above relations into \eqref{vmb} and drop the tilde, we can obtain that
\begin{align}
\label{vmbtomhd}
\begin{cases}
\epsilon \partial_t   f^+_\epsilon + \vdot f^+_\epsilon +  (  \epsilon E_\epsilon +    v\timess B_\epsilon) \sdot\nabla_v f^+_\epsilon =  \reps \Q(f^+_\epsilon,f^+_\epsilon) + \reps \Q(f^+_\epsilon,f^-_\epsilon), \\
\epsilon \partial_t   f^-_\epsilon + \vdot f^-_\epsilon - (  \epsilon E_\epsilon +    v\timess B_\epsilon) \sdot\nabla_v f^-_\epsilon =  \reps \Q(f^-_\epsilon,f^+_\epsilon) + \reps \Q(f^-_\epsilon,f^-_\epsilon), \\
\epsilon\partial_t   E_\epsilon - \curl   B_\epsilon = - \tfrac{ 1}{\epsilon^2} \intv (f_\epsilon^+ - f_\epsilon^-) v \bd v,\\
\partial_t   B_\epsilon + \curl   E_\epsilon =0,\\
\divg   B_\epsilon =0,~~ \divg   E_\epsilon = \tfrac{1}{\epsilon^2}\intv (f^+_\epsilon - f^-_\epsilon) \bdv.
\end{cases}
\end{align}
Before comparing the difference between \eqref{vmbtomhd} and the dimensionless VMB system in \cite[Equ.(3.10) ]{jama2012siam}, denoting
\[ F_\epsilon = f^+_\epsilon + f^-_\epsilon,  H_\epsilon= f^+_\epsilon - f^-_\epsilon,  \]
then \eqref{vmbtomhd} can be rewritten as
\begin{align}
\label{vmbsum}
\begin{cases}
\epsilon\partial_t F_\epsilon + \vdot F_\epsilon +  ( \epsilon   E_\epsilon +    v\timess   B_\epsilon) \sdot\nabla_v  H_\epsilon = \reps \Q(F_\epsilon,F_\epsilon) \\
\epsilon \partial_t  H_\epsilon + \vdot  H_\epsilon +  ( \epsilon   E_\epsilon +   v\timess   B_\epsilon) \sdot\nabla_v F_\epsilon = \reps \Q( H_\epsilon,F_\epsilon)  \\
\epsilon\partial_t   E_\epsilon - \curl   B_\epsilon = - \repst   \intv  H_\epsilon v \bd x, \\
  \partial_t    B_\epsilon + \curl   E_\epsilon =0,\\
\divg  \ B_\epsilon =0,~~\divg   E_\epsilon = \repst \intv  H_\epsilon \bd v.
\end{cases}
\end{align}
From \cite{jama2012siam},  the counterpart of \eqref{vmbsum} in their work is as follows
\begin{align}
\label{vmbjm}
\begin{cases}
\epsilon\partial_t F_\epsilon + \vdot F_\epsilon +  ( \epsilon   E_\epsilon +    v\timess   B_\epsilon) \sdot\nabla_v  G_\epsilon = \reps \Q(F_\epsilon,F_\epsilon) \\
\epsilon \partial_t  G_\epsilon + \vdot  G_\epsilon + \repst ( \epsilon   E_\epsilon +   v\timess   B_\epsilon) \sdot\nabla_v F_\epsilon = \reps \Q( G_\epsilon,F_\epsilon)  \\
\epsilon \mu_0 \partial_t   E_\epsilon - \curl   B_\epsilon = -  \mu_0 \intv  G_\epsilon v \bd x, \\
  \partial_t    B_\epsilon + \curl   E_\epsilon =0,\\
\divg  \ B_\epsilon =0,~~\divg   E_\epsilon = \reps \intv  G_\epsilon \bd v.
\end{cases}
\end{align}
In the above system, $G_\epsilon$ is the difference of $f_\epsilon^+$ and $f_\epsilon^-$. Compared to the equations of $G_\epsilon$ in \eqref{vmbjm} and $H_\epsilon$ in \eqref{vmbsum},  there exists one more $\repst$ before the Lorentz force in the equation of $G_\epsilon$. In Sec.\ref{sec-results}, we shall show  the Lorentz force in  system \eqref{vmbtomhd} is also very singular and barely  controlled. With one more $\repst$, the method of this work completely fails. The rigorous justification of \eqref{vmbjm} remains open.

The goal of this work is to derive the MHD system \eqref{mhddd} from the dimensionless system \eqref{vmbtomhd} in the classic solution framework. The key ingredient of this work is the uniform estimates of solutions to \eqref{vmbtomhd} with respect to the Knudsen number.  This is the first rigorous derivation work of MHD system from VMB equations.  This work can be seen as  kinetic approach to \cite{twonstomhd} where the MHD system is derived rigorously from the two fluid Navier-Stokes-Maxwell system in weak solution framework. We follows the ``mixed norm'' strategies used in used in \cite{briant-2015-be-to-ns,mouhotneumann-2006-decay} to obtain the uniform estimates. The advantage of this framework is to recover the dissipative estimates of  macroscopic part within less pages. But their strategies can not directly used to this work. Indeed, the works \cite{briant-2015-be-to-ns,mouhotneumann-2006-decay} are just for Boltzmann equation. In this work, we consider the VMB system. There exists very singular Lorentz force term. Furthermore, the norm used in \cite{briant-2015-be-to-ns} is anisotropic. It makes the Lorentz term harder to be controlled.  More details will be explained in Remark \ref{remark-diff} and Sec. \ref{sec-difficulty}.

The  rest part of this work is arranged as follows. Section \ref{sec-settings-main} is devoted to the notations, the assumption of the linear Boltzmann operator and the assumption of the initial data. The main results and the difficulties of this work  can be found in Section \ref{sec-results}. The Section \ref{secEstimates} consists in deducing the uniform estimates of solutions. The MHD limit of \eqref{vmbtomhd} is verified in Sec.\ref{sec-limit}.

\section{Preliminaries}
\label{sec-settings-main}

Firstly, we try to deduce the fluctuations system of  \eqref{vmbtomhd}.
Letting $\dm $ be the global Maxwellian, under the Navier-Stokes scalings for $G_\epsilon$ and $F_\epsilon$,
\[ F_\epsilon = \m ( 1 + \epsilon f_\epsilon ),~~  H_\epsilon = \epsilon \m  h_\epsilon, \]
and denoting
\[  j_\epsilon = \reps \intv h_\epsilon v \bdv,~n_\epsilon = \intv h_\epsilon \bdv,  \]
we can infer  the fluctuation system:
\begin{align}
\label{vmbtwolinear}
\begin{cases}
\partial_t f_\epsilon +  \reps \vdot f_\epsilon  +   \repst \mathcal{L}(f_\epsilon)  = E_\epsilon\sdot v \sdot    h_\epsilon -  (   E_\epsilon + \reps v \times B_\epsilon) \sdot \nabla_v    h_\epsilon + \reps \Gamma(f_\epsilon, f_\epsilon), \\
\partial_t  h_\epsilon +  \reps \vdot  h_\epsilon   - \repst  E_\epsilon \sdot v +   \repst \mathsf{L}( h_\epsilon)  =E_\epsilon\sdot v \sdot    f_\epsilon -  (   E_\epsilon + \reps v \times B_\epsilon) \sdot \nabla_v    f_\epsilon + \reps \Gamma(h_\epsilon, f_\epsilon), \\
\epsilon \partial_t E_\epsilon - \curl B_\epsilon = -   j_\epsilon , \\
   \partial_t B_\epsilon + \curl E_\epsilon =0,\\
\divg B_\epsilon =0,~~\epsilon \sdot \divg E_\epsilon =  n_\epsilon,
\end{cases}
\end{align}
where the linear Boltzmann operator $\mathcal{L}$ and $\mathsf{L}$ are defined as follows
\begin{align*}
- \mathsf{L}(w)&  = (\m)^{-1}\Q(\m w, \m), ~-\mathcal{L}(w)  = (\m)^{-1}\left( \Q(\m w, \m) + \Q(\m, \m  w)\right).
\end{align*}

By simple computation,  \eqref{vmbtwolinear} satisfies the following global conservation laws:
\begin{align}
\label{conservationlaws}
\begin{split}
\tdt \intt \left(u_\epsilon +  \epsilon E_\epsilon \times B_\epsilon\right)(t) \bd x = 0,\\
  \tdt \intt \left(\theta_\epsilon +   \epsilon \sdot \tfrac{ \epsilon |E_\epsilon|^2 + |B_\epsilon|^2}{3} \right)(t) \bd x  =0,\\
   \tdt \intt \rho_\epsilon(t) \bd x = \tdt \intt n_\epsilon(t) \bd x =0,~~\tdt \intt B_\epsilon(t) \bd x= 0,
\end{split}\end{align}
where
\[\rho_\epsilon = \intv f_\epsilon \bdv,~~u_\epsilon = \intv f_\epsilon v \bdv,~~\theta_\epsilon = \intv f_\epsilon \left( \tfrac{|v|^2-3}{3}\right) \bdv,~~ n_\epsilon = \intv h_\epsilon \bdv.   \]
\subsection{Notations and Terms}
 $\nabla^i_x f = \partial_{{x_1}}^{i_1} \partial_{x_2}^{i_2} \partial_{x_3}^{i_3} f$({\small $ \sum\limits_{k=1}^3i_k = i$})  is   the $i$-$th$ derivative  of $f$ with respect to $x$ .  we denote by $\nabla_x f$  the gradient of scalar function $f$.    $\nabla_v^i f$ and $\nabla_v f$ can be defined in the same way.  The Sobolev norm of $f$ are defined like this:
\begin{align*}
&\|f\|_{L^2_v}^2 =  \intv f^2 \bdv ,   ~~\|f\|_{L^2}^2 =  \intps f^2 \bdv \bd x,\\
& \|f\|_{H^s_x}^2 =  \sum\limits_{k=0}^s\|\nabla^k_x f\|_{L^2}^2,  \|f\|_{H^s}^2 =  \sum\limits_{k=0}^s \sum\limits_{i + j =k}\|\nabla^i_x \nabla^j_v f\|_{L^2}^2.
\end{align*}
Denoting $\hat{v}= \sqrt{1 + |v|}$,   the norms with weight on $v$ are defined as follows:
\begin{align*}
& \|f\|_{L^2_\Lambda}^2 = \|f\hat{v}\|_{L^2}^2,~ \|f\|_{H^s_{\Lambda_x}}^2 =  \sum\limits_{k=0}^s\|\nabla^k_x f {\hat{v}}\|_{L^2}^2,\\
& \|f\|_{\llv}^2 =  \intv f^2 \hat{v}^2\bdv
,~~\|f\|_{H^s_\Lambda}^2 =  \sum\limits_{k=0}^s \sum\limits_{i + j =k}\|\nabla^i_x \nabla^j_v f\|_{L^2_\Lambda}^2.
\end{align*}
Moreover, the positive constant $C$ is independent of $\epsilon$ and  different from lines to lines.
 $a \lesssim b$ means that there exists some positive constant $C$  such that $a \le C b$. We also use $C_0$ to indicate that the constant is dependent of the initial data.

\subsection{Assumption on the linear operators}
\label{sec-assump-on-L}

This section is on hypocoercivity
theory of the linear Boltzmann operator.  The assumptions in this subsection are the same to those in \cite{mouhotneumann-2006-decay} and \cite{briant-2015-be-to-ns}.  The verification of the assumption of those assumptions can be found in \cite[Sec.5.4]{mouhot-2006-homogeneous}.

\noindent {\bf H1}~( Coercivity and general controls.)  The Boltzmann operator $\mathcal{L}$ and $\mathsf{L}$ are self-joint operator from $L^2_v$ to $L^2_v$ with the following decomposition
\begin{align}
\label{assump-h1-h}
\mathcal{L} = -\mathbf{K} + \Lambda,~~ \mathsf{L} = -\mathbf{\Phi} + \Lambda,
\end{align}
where $\Lambda$ is a coercive operator. Furthermore, $\Lambda$ satisfies the following properties.
\begin{itemize}
\item  {For any $h, g \in L^2_v$,  there exists some $\lambda_0>0$ such that
\begin{align}
\label{assump-h-h-h}
\lambda_0 \|h\|_{\llv}^2 \le \intv \Lambda(h)\sdot h \bdv   \le C \|h\|_{\llv}^2,
\end{align}

and}
\begin{align}
\vert\intv  \Lambda(h)\sdot g\bdv  \vert \le C \|h\|_{\llv}\|g\|_{\llv}.
\end{align}
\item With respect to the derivative of $v$, the operator $\Lambda$ admits ``a defect of coercivity'', i.e.,
there exist some strictly  positive  constant  $1>\delta>0$ and $C_\delta$   such that
\begin{aligno}
\label{constant-a3-a4-hi}
\intps \nabla_{v}^i \nabla_x^j \Lambda(h)\sdot\nabla_{v}^i \nabla_x^j h\bdv \bd x \geqslant \delta \left\|\nabla_{v}^i \nabla_x^j h \right\|_{L^2_\Lambda}^{2}-C_\delta\|h\|_{H^{i+j-1}}^{2},~~i \ge 1.
\end{aligno}
\end{itemize}

\noindent {\bf H2}~( Mixing property in velocity.) This assumptions are about $\mathcal{L}$ and $\mathsf{L}$ .  For any $1>\delta>0$, there is some   constant $C_\delta>0$ such that
\begin{aligno}
&\vert \intps \nabla_{v}^i \nabla^j_x  \mathbf{K}(h) \sdot \nabla_{v}^i \nabla^j_x h\bdv  \vert+\vert \intps \nabla_{v}^i \nabla^j_x  \mathbf{\Phi}(h) \sdot \nabla_{v}^i \nabla^j_x h\bdv  \vert \\
&\quad  \leqslant C_\delta\|h\|_{H^{i+j-1}}^{2}+\delta\left\|\nabla_{v}^i \nabla^j_xh\right\|_{L^2_\Lambda}^{2},~~ i \ge 1.
\end{aligno}

\noindent {\bf H3} (Relaxation to the local equilibrium.)
The operators $\mathcal{L}$  and   $\mathsf{L}$ are closed and self-adjoint operators in $L^2_v$ space.
 Moreover, $$ \mathrm{Ker}\mathcal{L} = \mathrm{Span}\{1, v_1,v_2,v_3, \tfrac{|v|^2-3}{2}\},~~\mathrm{Ker}\mathcal{L} = \mathrm{Span}\{1\}.$$
Furthermore, $\mathcal{L}$ and $\mathsf{L}$ satisfy ``local coercivity assumption'':
\begin{align}
\label{a2}
\begin{split}
&  \intv \mathcal{L}(g) \sdot g \bdv \ge  \|g - \bdp g\|_{\llv}^2,   \intv \mathsf{L}(h) \sdot h \bdv \ge  \|h - \bdp h\|_{\llv}^2,
\end{split}
\end{align}
where $\bdp$ is the projection operator of $\mathcal{L}$ and $\mathsf{L}$ onto their kernel space respectively, i.e.,
\begin{aligno}
\bdp g = \intv g \bdv + v\sdot \intv g v \bdv + \tfrac{|v|^2 -3}{2} \intv g \tfrac{|v|^2 -3}{3} \bdv,~~\bdp h = \intv h \bdv.
\end{aligno}
Besides, we also assume that
\begin{align}
\vert \intv f \sdot \mathcal{L}(g) \bdv \vert  \le C \|f\|_{\llv} \|g\|_{\llv},~~\vert \intv f \sdot \mathsf{L}(g) \bdv \vert  \le C \|f\|_{\llv} \|g\|_{\llv},~~\forall f,~g \in L^2_{\llv}.
\end{align}

\noindent {\bf H4}~(Control on the second order operator.) This assumption is on $\Gamma(g,g)$ and $\Gamma(g,h)$.
\begin{itemize}
\item  {For any $g, h\in L^2$,  $\Gamma(g,g) \in \mathrm{Ker}(\mathcal{L})^\perp$,~~ $\Gamma(g,h) \in \mathrm{Ker}(\mathsf{L})^\perp$.}
\item  {For the non-linear operator and $s \ge 3$  }
\begin{align}
\label{constant-cn}
\begin{split}
&\vert \intps \nabla^s_x \Gamma(g,h) \sdot f \bdv \bd x \vert \lesssim \|(g,h)\|_{H^s_{x}} \|(g,h)\|_{H^s_{\Lambda_x}}\|f^\perp\|_{L^2_\Lambda},\\
&\vert \intps \nabla^j_x \nabla_v^i \Gamma(g,h) \sdot f \bdv \bd x \vert \lesssim \|(g,h)\|_{H^s} \|(g,h)\|_{H^s_{\Lambda }}\|f\|_{L^2_\Lambda},~~i\ge 1,~~ s = i+j.
\end{split}
\end{align}
\end{itemize}

\begin{remark}
\label{remark-one}
 In the general case, the lower bound $\lambda_0$ is determined by the collision frequency in  \eqref{assump-h-h-h}. To avoid for using too many notations and without loss of generality, we assume the lower bound $\lambda_0$ to be one in the rest part of this paper.

\end{remark}

\subsection{Assumption on the initial data}
Recalling
\begin{align*}
&\rho_\epsilon = \intv f_\epsilon \bdv,~ u_\epsilon = \intv f_\epsilon v\bdv,~\theta_\epsilon = \intv \left( \tfrac{|v|^2}{3} -1\right) f_\epsilon \bdv,~n_\epsilon = \intv  h_\epsilon \bdv,
\end{align*}
similar to \cite{briant-2015-be-to-ns,guo-2003-vmb-invention}, we can assume the initial data
\begin{align}
\label{estMeaninitial}
\begin{split}
\intt \left(u_\epsilon +  \epsilon E_\epsilon \times B_\epsilon\right)(0) \bd x = 0,\\
  \intt \left(\theta_\epsilon +   \epsilon \sdot \tfrac{ \epsilon |E_\epsilon|^2 + |B_\epsilon|^2}{3} \right)(0) \bd x  =0,\\
    \intt \rho_\epsilon(0) \bd x = \intt n_\epsilon(0) \bd x =0,~~\intt B_\epsilon(0) \bd x= 0.
\end{split}
\end{align}
The assumption \eqref{estMeaninitial} means that the initial data of \eqref{vmbtomhd} are with the same mass, total momentum and energy to the steady state $(\m, \m, 0, 0)$. Furthermore, from \eqref{conservationlaws}, we can infer that the solution preserves these properties all the time.

\section{ Main results}
\label{sec-results}

Define
\begin{align*}
\mathcal{H}_\epsilon^s(t) := \|(f_\epsilon,  g_\epsilon, B_\epsilon, \sqrt\epsilon E_\epsilon)\|_{H^s_x}^2 + \epsilon^2\|(\nabla_v f_\epsilon, \nabla_v  g_\epsilon)\|_{H^{s-1}}^2,\\
\mathcal{D}_\epsilon^s(t): = \|(   f_\epsilon,     h_\epsilon)\|_{H^{s}_\Lambda}^2 +   \|(E_\epsilon, B_\epsilon)\|_{H^{s-1}_{ x}}^2 +  \repst \|(   f_\epsilon^\perp,     h_\epsilon^\perp)\|_{H^{s}_{\Lambda_x}}^2 + \reps \|n_\epsilon\|_{H^{s-1}}^2.
\end{align*}

\begin{theorem}
\label{theoremexi}
Under the assumption  in the section \ref{sec-assump-on-L} and the assumption \eqref{estMeaninitial} on the initial data,    there exists some small enough constant $c_0$  such that  if  the initial data $(f_\epsilon(0),  g_\epsilon(0), E_\epsilon(0), B_\epsilon(0))$ satisfy
\[{H}_\epsilon^s(0) \le c_0,~~s \ge 3,  \]
then system \eqref{vmbtwolinear} admit a unique global classic solution $(f_\epsilon,  h_\epsilon, B_\epsilon, E_\epsilon)$  satisfying for any $t>0$
\begin{align}
\label{esttheowhole}
\sup\limits_{ 0 \le s \le t}  {H}_\epsilon^s(t) +   \tfrac{1}{2} \int_0^t \mathcal{D}^s_\epsilon(\tau)\bd \tau \le \tfrac {c_u}{c_l}  {H}_\epsilon^s(0),~~\forall \epsilon \in  (0,1],
\end{align}
where $c_l$ and $c_u$ are positive constants only dependent of the Sobolev embedding constant.
\end{theorem}

\begin{remark}
We  use a equivalent norm $\tilde{H}_\epsilon^s$ (\eqref{normequalhs}) instead of $H^s_\epsilon$ to obtain the prior estimate \eqref{esttheowhole}. The constants $c_l$ and $c_u$ come from the equivalent relation of $\tilde{H}_\epsilon^s$ and $H^s_\epsilon$, see \eqref{estclcu}. Furthermore, since we have set the lower bound in \eqref{assump-h-h-h} to be one (see Remark \ref{remark-one}), in the general case, $c_0$ is dependent of the lower bound in \eqref{assump-h-h-h}.

\end{remark}

\begin{remark}
\label{remark-diff}

Noticing in the definition of $\mathcal{D}_\epsilon^s$, we lose order one derivative (with $x$) of electromagnetic field. But in the Lorentz term, the exists one extra order one derivative (with $v$). Owing to these two facts, the Lorentz force are not easy to bound. On the other hand, the key point of the proof is to obtain energy estimates like:\
\[  \tdt \mathcal{H}_\epsilon^s  + \mathcal{D}_\epsilon^s \lesssim \mathcal{D}^s_\epsilon(t) {\mathcal{H}_\epsilon^s}(t).  \]
Noticing that there exists $\sqrt{\epsilon}$ before $E_\epsilon$ in $\mathcal{H}_\epsilon^s$, it brings difficulties during the proof. Taking the first equation of \eqref{vmbtomhd} for example, during the energy estimates, there will exist  terms like $-\intps \nabla_x^s E_\epsilon    \sdot \nabla_v ( \m h_\epsilon) \nabla_x^s f_\epsilon  \bd v \bd x$ which is not easy to bound by $\hd$. We must split $\epsilon$ from $-\intps \nabla_x^s E_\epsilon    \sdot \nabla_v ( \m h_\epsilon) \nabla_x^s f_\epsilon  \bd v \bd x$ to close the energy estimates.  See Sec. \ref{sec-difficulty} for more details and strategies.

Furthermore, from \cite[pp.19]{diogosrm-2019-vmb-fluid}(where $\alpha =\epsilon,~~\beta=\gamma=1$), the dimensionless VMB system(for strong interaction $\delta=1$) to Navier-Stokes-Maxwell  system is as follows:
  \begin{align}
  \label{vmbtnsw}
  \begin{cases}
  \epsilon \partial_t   f^+_\epsilon + \vdot f^+_\epsilon +  (  \epsilon E_\epsilon +    v\timess B_\epsilon) \sdot\nabla_v f^+_\epsilon =  \reps \Q(f^+_\epsilon,f^+_\epsilon) + \reps \Q(f^+_\epsilon,f^-_\epsilon), \\
  \epsilon \partial_t   f^-_\epsilon + \vdot f^-_\epsilon - (  \epsilon E_\epsilon +    v\timess B_\epsilon) \sdot\nabla_v f^-_\epsilon =  \reps \Q(f^-_\epsilon,f^+_\epsilon) + \reps \Q(f^-_\epsilon,f^-_\epsilon), \\
  \partial_t   E_\epsilon - \curl   B_\epsilon = - \tfrac{ 1}{\epsilon^2} \intv (f_\epsilon^+ - f_\epsilon^-) v \bd v,\\
  \partial_t   B_\epsilon + \curl   E_\epsilon =0,\\
  \divg   B_\epsilon =0,~~ \divg   E_\epsilon = \tfrac{1}{\epsilon}\intv (f^+_\epsilon - f^-_\epsilon) \bdv.
  \end{cases}
  \end{align}
  Although the kinetic parts of \eqref{vmbtnsw} and \eqref{vmbtomhd}  are the same, the Lorentz force in \eqref{vmbtomhd} are harder to be bounded than that in \eqref{vmbtnsw}. Indeed, there exists $\epsilon$ before $\partial_t E_\epsilon$ in the third equation of \eqref{vmbtomhd}. Due to this extra coefficent $\epsilon$, there is no useful estimate of electric field. This makes the Lorentz term are harder to be bounded. See Sec. \ref{sec-difficulty} for more details.
\end{remark}

Before  stating the theorem on fluid limit,  we introduce the following $A(v)$ and $B(v)$, vector $\tilde{v}$
\begin{aligno}
\label{estab}
A(v) = v \otimes v - \tfrac{|v|^2}{3}\mathbb{I},~~B(v) = v ( \tfrac{|v|^2}{2} - \tfrac{5}{2}), ~\mathcal{L}\hat{A}(v) = A(v),~~\mathcal{L}\hat{B}(v) = B(v),~~\mathsf{L}\tilde{v} = v.
\end{aligno}
Then, denoting
\begin{align}
\label{nusigma}
 \nu = \tfrac{1}{15} \sum\limits_{ 1 \le i \le 3 \atop 1 \le j \le 3}\intv A_{ij}\hat{A}_{ij}\bdv,~~\kappa = \tfrac{2}{15} \sum\limits_{ 1 \le i \le 3  }\intv B_{i}\hat{B}_{i}\bdv,~~\sigma = \tfrac{1}{3}\intv \tilde{v}\sdot v \bdv,
\end{align}
these three strictly positive constants will appear in the MHD system. Furthermore, let $u_0, \theta_0, n_0, E_0, B_0 \in H^s_x$ and satisfy (up to a subsequence)
\[ \mathbf{P} u_\epsilon(0) \to u_0,~~\tfrac{3}{5}\theta_\epsilon(0) - \tfrac{2}{5}\rho_\epsilon(0)  \to \theta_0,~B_\epsilon(0) \to B_0, \text{in}~~H^{s-1}_x.\]
where $\mathbf{P}$ is the Leray projector.
\begin{theorem}[Fluid limit]
\label{theoremlimit}
Under the assumption  in the section \ref{sec-assump-on-L} and the assumption \eqref{estMeaninitial} on the initial data, for  the solutions $f_\epsilon,~ h_\epsilon,~E_\epsilon,~B_\epsilon$ constructed in Theorem \ref{theoremexi}, it follows that  for any $T >0$
\begin{aligno}
\label{esttheolimit}
f_\epsilon \to \rho + u \sdot v + \tfrac{|v|^2-3}{2} \theta, ~~  h_\epsilon \to 0,~~B_\epsilon \to B,~~\text{in}~~~L^2(0,T);H^{s-1}_x),
\end{aligno}
where $\rho,~u,~\theta,~B$(belonging to $L^\infty((0,\infty);H^s_x)$) are strong solutions to the following MHD system:
\begin{align*}
\begin{cases}
\partial_t u + u\sdot \nabla u - \nu \Delta u + \nabla P = (\curl B) \timess B,\\
\partial_t \theta + u \sdot\nabla \theta - \kappa \Delta \theta =0,\\
\divg u = \divg B= 0,~~ \rho + \theta =0,\\
\partial_t B  - \tfrac{1}{\sigma} \Delta B = \curl(u \times B),\\
u(0)=u_0,~\theta(0)=\theta_0,~B(0)=B_0.
\end{cases}
\end{align*}
Furthermore, for any $\tau>0$, we can infer
\begin{align}
\label{well}
\mathbf{P} u_\epsilon \to u,~~\tfrac{3}{5}\theta_\epsilon - \tfrac{2}{5}\rho_\epsilon \to \theta,~~\text{in},~~C([\tau, +\infty); H^{s-1}_x).
\end{align}

\end{theorem}

\begin{remark}
\label{remark-ohm}
The derivation of system \eqref{mhddd} is based on the approximate conservation laws. Here, we comments on the Ohm's law.  From  Remark \ref{remark-diff}, the kinetic parts of \eqref{vmbtomhd} and \eqref{vmbtnsw} are the same,  in \cite{mvmbtothree} \cite[pp.61]{diogosrm-2019-vmb-fluid},  the Ohm's law derived from \eqref{vmbtnsw} is
\[ j =  \sigma( E + u \times B + \nabla n) - nu. \]
Noticing that there exists a coefficient $\reps$ before the dissipative energy estimates of $n_\epsilon$ in the definition of $\mathcal{D}_\epsilon^s$, for \eqref{vmbtomhd}, we can infer that
\[  n=0. \]
From the third equation of \eqref{vmbtomhd}, as Knudsen number goes to zero, we can infer that   $$j = \curl B.$$
All together, we can  verify the Ohm’s law for MHD system
\[ j_\epsilon \to j= \curl B=\sigma (   E  +  u\timess B),~~~~\text{in the distributional sense}.\]
This is how we can recover the magnetic field equation in \eqref{mhddd}.

\end{remark}

\begin{remark}
If the initial data are well-prepared, i.e.,
\[  f_\epsilon(0)= \rho_0 + u_0 \sdot + \tfrac{|v|^2-3}{2} \theta_0,~~ \divg u_0 =0,~~ \rho_0 + \theta_0 =0,  \]
then \eqref{well}
can be improved to
\begin{align*}
\mathbf{P} u_\epsilon \to u,~~\tfrac{3}{5}\theta_\epsilon - \tfrac{2}{5}\rho_\epsilon \to \theta,~~\text{in},~~C([0, +\infty); H^{s-1}_x).
\end{align*}
\end{remark}

\subsection{Historical background, difficulties and novelty }
\label{sec-difficulty}

\subsubsection{Historical background strategies} There are two ways of justifying  the hydrodynamics limit of the Boltzmann equation and its coupled system. One is based on the  the renormalized solution.  We refer to \cite{bgl1993convergence,diperna-lions1989cauchy,gsrm2004,lm2010soft,lm2001acoustic,masmoudi-srm2003stokesfourier,mischler2010asens} for the  work on the existence of renormalized solutions and fluid limit in renormalized solutions work. The other is in the classic   solution  framework.   The existence of classic solution to VMB system can be found in \cite{dlyz2017cmp,guo-2003-vmb-invention,sr2006-vmb}. Basically, there are three strategies of verifying rigorously the fluid limits of the Boltzmann equation and its coupled system: spectral analysis of the semi-group(see \cite{bu1991,vpb2020limit-spectrum}) , Hilbert expansion methods (see \cite{guo2006NSlimit,twovpblimits,vmbtonsp,vmbtonswh}) and convergence method based on uniform estimates (see \cite{ns-limit-2018,briant-2015-be-to-ns,vmbtonswu,uvpb2020}).  The Navier-Stokes limit of the Boltzmann equation can be found in \cite{bu1991,briant-2015-be-to-ns,guo2006NSlimit,ns-limit-2018}.    The diffusive limit of the Valsov-Poisson-Boltzmann(VPB) equation was inverstigated in \cite{uvpb2020,jz2020vpbconvergence,vpb2020limit-spectrum,twovpblimits}.

As mentioned before in the Introduction, the key step towards to the justification is the uniform estimates. From the  local coercivity properties of the linear Boltzmann operators (\eqref{a2}), there is only dissipative estimates for the microscopic parts ( $f_\epsilon - \bdp f_\epsilon$ and $h_\epsilon - \bdp h_\epsilon$). To get inequalities like
\[ \tdt \mathcal{H}^s_\epsilon(t) + \mathcal{D}_\epsilon^s \le \mathcal{H}_\epsilon^s(t) \mathcal{D}_\epsilon^s(t), \]
we must obtain  the dissipative estimates of the macroscopic parts  $\bdp f_\epsilon $ and $\bdp h_\epsilon$. To achieve this, one idea is to employ the  Grad’s 13 moment equations (see \cite{vmbtonswu,guo-2003-vmb-invention} for instance). In this work, we follow the ``mixed norm'' idea used in \cite{briant-2015-be-to-ns,mouhotneumann-2006-decay,mvmbtothree}. From the point view of this work, the important ingredient of this ``mixed norm $ \dt \intps \nabla_x f_\epsilon \sdot \nabla_v f_\epsilon \bdv \bd x$'' is the following inequality: taking the first equation in \eqref{vmbtwolinear} for example
\begin{align}
\label{est-z2}
  \dt \intps \nabla_x f_\epsilon \sdot \nabla_v f_\epsilon \bdv \bd x +   \reps \|\nabla_x f_\epsilon\|_{L^2}^2 \le \cdots.
\end{align}
In the above equation, there exists dissipative estimates for the macroscopic part. The advantage of this framework is that we can recover the macroscopic parts in one simple inequality to avoid using the Grad's 13 moment equations which is quite involved and  hard to bound for the dimensionless VMB system (see the almost one hundred pages work \cite{vmbtonswu} on Navier-Stokes-Maxwell limit of VMB system). But for the dimensionless system, the norm used in \cite{briant-2015-be-to-ns} is anisotropic. It brings new difficulties in the process of bounding the singular Lorentz term.

\subsubsection{difficulties}
The difficulties in obtaining the uniform estimates  come from the very singular Lorentz force term and the hyperbolicity of Maxwell's system.  Firstly, recalling the instant energy norm $\mathcal{H}^s_\epsilon$
\[\mathcal{H}_\epsilon^s(t)= \|(f_\epsilon,  g_\epsilon, B_\epsilon, \sqrt\epsilon E_\epsilon)\|_{H^s_x}^2 + \epsilon^2\|(\nabla_v f_\epsilon, \nabla_v  g_\epsilon)\|_{H^{s-1}}^2, \]
  there is no $L^\infty$ bound of $f_\epsilon$ and $h_\epsilon$ on the phase space. Besides, there
also no useful $L^2$ estimates of $(\nabla_v^i f_\epsilon, \nabla_v^i h_\epsilon)(i \ge 1)$ with respect to time $t$. Furthermore, since there exists $\sqrt \epsilon$ before $E_\epsilon$,  there is no useful estimate of the electric field can be derived from the anisotropic norm $\mathcal{H}_\epsilon^s$ too.

Due to  this anisotropic norm $\mathcal{H}_\epsilon^s$ and the singular Lorentz force term,  there exist difficulties in the process of obtaining uniform estimates. Indeed, recalling the equations of $f_\epsilon$ and $h_\epsilon$
\begin{align}
\label{vmbtwolinearc}
\begin{cases}
\partial_t f_\epsilon +  \reps \vdot f_\epsilon  +   \repst \mathcal{L}(f_\epsilon)  = E_\epsilon\sdot ( v \sdot    h_\epsilon - \nabla_v    h_\epsilon)    - \reps (v \times B_\epsilon) \sdot \nabla_v    h_\epsilon + \reps \Gamma(f_\epsilon, f_\epsilon), \\
\partial_t  h_\epsilon +  \reps \vdot  h_\epsilon   - \repst  E_\epsilon \sdot v +   \repst \mathsf{L}( h_\epsilon)  =E_\epsilon\sdot ( v \sdot    f_\epsilon - \nabla_v    f_\epsilon)    - \reps (v \times B_\epsilon) \sdot \nabla_v    f_\epsilon + \reps \Gamma(h_\epsilon, f_\epsilon), \\
\epsilon \partial_t E_\epsilon - \curl B_\epsilon = -   j_\epsilon , \\
   \partial_t B_\epsilon + \curl E_\epsilon =0,\\
\divg B_\epsilon =0,~~\epsilon \sdot \divg E_\epsilon =  n_\epsilon,
\end{cases}
\end{align}
the first two terms on the right hand of the first two equations in \eqref{vmbtwolinearc} are generated by the Lorentz force term. Since there exists extra $\reps$ before $(v \times B_\epsilon) \sdot \nabla_v    h_\epsilon $, there exists difficulty in bounding  this term with magnetic field.  Formally, the term $E_\epsilon \sdot \nabla_v h_\epsilon$ is not singular. But it is very hard to bound. Indeed, as mentioned before, no $L^\infty$ bound of $f_\epsilon$ and $h_\epsilon$, no $L^2$ bound of $\nabla_v f_\epsilon$ and $\nabla_v h_\epsilon$ and no useful bound of $E_\epsilon$ are at our disposal.

  The idea of dealing with these difficulties  goes like this. By decomposing $f_\epsilon$ and $h_\epsilon$ into macroscopic part and microscopic part, we can obtain the $L^\infty$ bound of the macroscopic part ($\rho_\epsilon,~u_\epsilon,~\theta_\epsilon,~n_\epsilon$) from $\mathcal{H}_\epsilon^s$. For the microscopic part, since there exists coefficent $\repst$ before the microscopic part in the definition of $\mathcal{D}_\epsilon^s$ (only with derivation to $x$), the singular Lorentz force term can be bounded by virtue of integration by part and the structrue of the Maxwell's system.

Since the Maxwell's equations are hyperbolic, to close the energy estimates,   the dissipative estimates  of the eletromagnetic field are needed. This difficulty can be overcome by employing the idea used in \cite{mvmbtothree}. The idea is to employ the equation of $g_\epsilon$ to obtain a new equation containing a damping term of $E_\epsilon$.  In details, multiplying the equation of $g_\epsilon$ by $\tilde{v}$(see \eqref{estab} and $\tilde j_\epsilon = \reps \intv h_\epsilon \tilde v \bdv$) and then integrating over $\mathbb{R}^3$, we can obtain that
\begin{align}
\label{idea}
-   \partial_t \tilde{j}_\epsilon + \cdots + \tfrac{ 1 }{\epsilon^2} E_\epsilon = \cdots.
\end{align}
From this equation, the dissipative energy estimate of the electromagnetic field can be obtained.

Based on the uniform estimates, the MHD system can be obtained by employing the local conservation laws of system \eqref{vmbtwolinear}. The idea of recovering the equation of $B$ in \eqref{mhddd} is to employ  the Ohm's law derived from the dimensionless VMB system. Indeed, from \eqref{vmbtomhd}, we can finally obtain that
\[j= \curl B = \sigma ( E + u \times B) .\]
Based on the above relation, the last equation in \eqref{mhddd} can be obtained(see Remark \ref{remark-ohm} for more details).

\section{A prior estimates}
\label{secEstimates}
This section is devoted to proving the existence of solutions to \eqref{vmbtwolinear}, i.e., Theorem \ref{theoremexi}.  The key ingredient is the uniform prior estimate of solutions.  The proof is quite involved. We split the whole proof into four lemmas.

\begin{align}
\label{vmbtwo-rewrite}
\begin{cases}
\partial_t f_\epsilon +  \reps \vdot f_\epsilon  -   \repst \mathcal{L}(f_\epsilon)  = - \lrf \sdot \nabla_v (\m h_\epsilon) + \reps\Gamma(f_\epsilon,f_\epsilon)=N_1, \\
\partial_t h_\epsilon +  \reps \vdot h_\epsilon   - \reps  E_\epsilon \sdot v -   \repst \mathsf{L}(h_\epsilon)  =  - \lrf \sdot \nabla_v (\m f_\epsilon) + \reps\Gamma(h_\epsilon,f_\epsilon)=N_2, \\
\epsilon\partial_t E_\epsilon - \curl B_\epsilon = -   j_\epsilon, \\
  \partial_t B_\epsilon + \curl E_\epsilon =0,\\
\divg B_\epsilon =0,~~ \epsilon \sdot \divg E_\epsilon = n_\epsilon.
\end{cases}
\end{align}

\subsection{The dissipative estimates of the microscopic part}

\begin{lemma}[only related to $\nabla_x^k$]
\label{lemmaonlyx}
Under the assumptions  of Theorem \ref{theoremexi},  if $(f_\epsilon,  h_\epsilon, B_\epsilon, E_\epsilon)$ are strong solutions to \eqref{vmbtwolinear}, then
\begin{aligno}
\label{estlemmaonlyx}
& \dt \|  (f_\epsilon,  h_\epsilon, \seps E_\epsilon, B_\epsilon)\|_{H^s_x}^2  + \tfrac{1}{\epsilon^2}   \| ( f^\perp_\epsilon,  h^\perp_\epsilon)\|_{H^s_{\Lambda_x}}^2     \lesssim \hd.
\end{aligno}

\end{lemma}

\begin{proof}
Applying $\nabla_x^k $ to the first four equations of \eqref{vmbtwo-rewrite} and then multiplying the resulting equations by $\nabla^k_x f_\epsilon \m $, $\nabla_x^k  h_\epsilon \m $, $\nabla_x^k E_\epsilon \m$ and $\nabla_x^k B_\epsilon \m $, the integration over the phase space leads to
\begin{aligno}
\label{estx-0}
& \dt \|\nabla^k_x (f_\epsilon, h_\epsilon, \seps E_\epsilon, B_\epsilon)\|_{L^2}^2    - \repst  \intps  \left( \mathcal{L}(\nabla_x^k f_\epsilon)\sdot \nabla_x^k f_\epsilon    +  \mathsf{L}(\nabla_x^k  h_\epsilon)\sdot \nabla_x^k h_\epsilon \right) \bdv \bd x\\
& =   -\intps \left( \nabla_x^k \left(    \lrf \sdot \nabla_v ( \m h_\epsilon)\right) \nabla_x^k f_\epsilon  +  \nabla_x^k \left(\lrf \sdot \nabla_v ( \m f_\epsilon) \right)\nabla_x^k h_\epsilon \right) \bdv \bd x \\
& + \reps \intps \left( \nabla_x^k   \Gamma(f_\epsilon,f_\epsilon)  \sdot \nabla_x^k f_\epsilon  +  \nabla_x^k \Gamma(h_\epsilon,f_\epsilon)  \sdot \nabla_x^k h_\epsilon \right)   \bdv \bd x \\
& := D_1 + D_n.
\end{aligno}
  Thus we only need to pay attention to the term with coefficient $\reps$. Secondly, while $k=s$ and all the derivative acts on $ h_\epsilon$ and $f_\epsilon$, that is to say,
\[ \intps \left(  \left(   \lrf\sdot  ( \m \nabla_x^k \nabla_v  h_\epsilon)\right) \nabla_x^k f_\epsilon  +    \left(\lrf \sdot  ( \m \nabla_x^k \nabla_v f_\epsilon) \right)\nabla_x^k  h_\epsilon \right) \bdv \bd x, \]
the above term can not be directly controlled. To overcome these difficulties, we first split
\begin{aligno}
\label{estonlyD1-0}
D_1 &  = \intps \left( \left(   \lrf \sdot \nabla_v  \nabla_x^k ( \m  h_\epsilon)\right) \nabla_x^k f_\epsilon    \right) \bdv \bd x \\
 & +  \intps \left(      \nabla_x^k \left(\lrf\sdot \nabla_v  \nabla_x^k ( \m f_\epsilon) \right)\nabla_x^k  h_\epsilon \right) \bdv \bd x + D_r \\
& := D_2 + D_r.
 \end{aligno}
$D_2$ is simple and can be bounded by integrating by parts over the phase space. Indeed, by simple computation,  we can conclude that
\begin{aligno}
\label{estonlyxD2h}
D_2 &   = -   \intps         v \sdot  E_\epsilon \nabla_x^k f_\epsilon     \nabla_x^k ( h_\epsilon  )  \bdv \bd x\\
 & =  -   \intps         v \sdot  E_\epsilon \nabla_x^k f_\epsilon     \nabla_x^k ( n_\epsilon  )  \bdv \bd x - \intps         v \sdot  E_\epsilon \nabla_x^k f_\epsilon     \nabla_x^k ( h_\epsilon^\perp  )  \bdv \bd x.
\end{aligno}
Noticing that while $s \ge 3$
\begin{align*}
\vert\intps         v \sdot  E_\epsilon \nabla_x^k f_\epsilon     \nabla_x^k ( n_\epsilon  )  \bdv \bd x  \vert \lesssim \| E _\epsilon\|_{H^{s-1}}\|f_\epsilon\|_{H^s_{\Lambda_x}} \|  h_\epsilon \|_{H^s_{x}} \le  \mathcal{D}^s_\epsilon(t) \sqrt{\mathcal{H}_\epsilon^s}(t),\\
\vert \intps         v \sdot  E_\epsilon \nabla_x^k f_\epsilon     \nabla_x^k ( h_\epsilon^\perp  )  \bdv \bd x\vert \lesssim\|\sqrt\epsilon E _\epsilon\|_{H^s}\|f_\epsilon\|_{H^s_{\Lambda_x}} \|\reps h_\epsilon^\perp\|_{H^s_{\Lambda_x}}\lesssim \mathcal{D}^s_\epsilon(t) \sqrt{\mathcal{H}_\epsilon^s}(t),
\end{align*}
combining \eqref{estonlyxD2h}, it follows that
\begin{aligno}
\label{estonlyxD2}
D_2  \lesssim \mathcal{D}^s_\epsilon(t) \sqrt{\mathcal{H}_\epsilon^s}(t)
\end{aligno}

For $D_2$, denoting  \begin{aligno}
\label{estxdr}
D_r &  = \sum\limits_{ i \ge 1 \atop i+j=k} \intps   \nabla_x^i\left(    \lrf\right) \sdot \nabla_v  \nabla_x^j ( \m  h_\epsilon) \nabla_x^k f_\epsilon      \bdv \bd x \\
 & + \sum\limits_{ i \ge 1 \atop i+j=k} \intps        \nabla_x^i \left(\tfrac{\epsilon E_\epsilon +   v\times B_\epsilon}{\m \epsilon}\right) \sdot \nabla_v  \nabla_x^j ( \m f_\epsilon) \nabla_x^k  h_\epsilon   \bdv \bd x\\
 & := D_{r1} + D_{r2},
 \end{aligno}
 and noticing that
\begin{align*}
 h_\epsilon(t,x,v) = n_\epsilon(t,x) +  h_\epsilon^\perp(t,x,v),~\nabla_v n(t,x)=0,~~
\end{align*}
we have that
\begin{aligno}
\label{estr222}
D_{r2} &  = \sum\limits_{ i \ge 1 \atop i+j=k} \intps        \nabla_x^i \lrf \sdot \nabla_v  \nabla_x^j ( \m f_\epsilon) \nabla_x^k  h_\epsilon   \bd v \bd x \\
& = \sum\limits_{ i \ge 1 \atop i+j=k} \intps        \nabla_x^i \lrf \sdot \nabla_v  \nabla_x^j ( \m  f_\epsilon) \nabla_x^k  h_\epsilon^\perp  \bd v \bd x\\
& = -  \sum\limits_{ i \ge 1 \atop i+j=k} \intps     v\sdot   \nabla_x^i (\epsilon E_\epsilon) \nabla_x^j  f_\epsilon  \nabla_x^k (\reps h_\epsilon^\perp)  \bdv \bd x \\
& +  \sum\limits_{ i \ge 1 \atop i+j=k} \intps        \nabla_x^i (\epsilon E_\epsilon)     \sdot \nabla_v \nabla_x^j  f_\epsilon \cdot \nabla_x^k  ( \reps h_\epsilon^\perp) \bdv \bd x \\
& +  \sum\limits_{ i \ge 1 \atop i+j=k} \intps       v\times \nabla_x^i B_\epsilon     \sdot \nabla_v \nabla_x^j  f_\epsilon \cdot \nabla_x^k ( \reps h_\epsilon^\perp ) \bdv \bd x\\
& =D_{r21} + D_{r22} + D_{r23}.
\end{aligno}
The three  $D_{r21}$, $D_{r22}$ and $D_{r23}$ can be bounded in the similar way. Taking $D_{r23}$ for example,
\begin{align*}
D_{r23} & \le   \sum\limits_{ i \ge 1 \atop i+j=k} \intt |\nabla_x^i B_\epsilon| \intv          |v|   |\nabla_v \nabla_x^j  f_\epsilon| \cdot |\nabla_x^k \reps h_\epsilon^\perp|  \bdv \bd x \\
&  \le   \sum\limits_{ [\tfrac{s}{2}] \ge i \ge 1 \atop i+j=k} \intt |\nabla_x^i B_\epsilon(t,x)|   \|\nabla_v \nabla_x^j  f_\epsilon(t,x)\|_{L^2_{\Lambda_v}}   \|  \nabla_x^k \reps h_\epsilon^\perp(t,x)\|_{L^2_{\Lambda_v}} \bd x \\
& +   \sum\limits_{ [\tfrac{s}{2}] \le i \le s \atop i+j=k} \intt |\nabla_x^i B_\epsilon(t,x)|   \|\nabla_v \nabla_x^j  f_\epsilon(t,x)\|_{L^2_{\Lambda_v}}   \|  \nabla_x^k \reps  h_\epsilon^\perp(t,x)\|_{L^2_{\Lambda_v}} \bd x \\
& \le    \|\nabla_v \nabla_x^j  f_\epsilon(t,x)\|_{L^2_{\Lambda_v}}\|_{L^\infty_x}  \sum\limits_{ [\tfrac{s}{2}] \le i \le s \atop i+j=k} \intt |\nabla_x^i B_\epsilon(t,x)|     \|  \nabla_x^k \reps h_\epsilon^\perp(t,x)\|_{L^2_{\Lambda_v}} \bd x \\
& +   \|\nabla_x^i B_\epsilon(t,x)\|_{L^\infty_x} \sum\limits_{ [\tfrac{s}{2}] \ge i \ge 1 \atop i+j=k} \intt    \|\nabla_v \nabla_x^j  f_\epsilon(t,x)\|_{L^2_{\Lambda_v}}   \|  \nabla_x^k  \reps h_\epsilon^\perp(t,x)\|_{L^2_{\Lambda_v}} \bd x\\
& \lesssim \|B_\epsilon\|_{H^s}\|f_\epsilon\|_{H^s_\Lambda}\|\reps h_\epsilon^\perp\|_{H^s_\Lambda}\\
& \lesssim \hd.
\end{align*}
All together, we can infer that
\begin{aligno}
\label{estonlyxD2r}
D_{r2} \lesssim \hd.
\end{aligno}
For left $D_{r1}$ in \eqref{estxdr}, we decompose $D_{r1}$ as follows
\begin{aligno}
\label{estonlyxD1r-1}
D_{r1} & =  \sum\limits_{ i \ge 1 \atop i+j=k} \intps   \nabla_x^i\left(   \lrf\right) \sdot \nabla_v  \nabla_x^j ( \m  h_\epsilon) \nabla_x^k f_\epsilon^\perp     \bd v \bd x \\
& + \sum\limits_{ i \ge 1 \atop i+j=k} \intps   \nabla_x^i \lrf \sdot \nabla_v  \nabla_x^j ( \m n_\epsilon) \nabla_x^k \bdp f_\epsilon     \bd v \bd x \\
& + \sum\limits_{ i \ge 1 \atop i+j=k} \intps   \nabla_x^i\lrf \sdot \nabla_v  \nabla_x^j ( \m  h_\epsilon^\perp) \nabla_x^k \bdp f_\epsilon     \bd v \bd x\\
&:= D_{r11} + D_{r12}+ D_{r13}.
\end{aligno}
By the same way of dealing with $D_{r2}$, we can infer that
\begin{aligno}
\label{estonlyxD1r-3}
|D_{r11}| + |D_{r13}| \lesssim \hd.
\end{aligno}
As for $D_{r12}$ in \eqref{estonlyxD1r-1},  recalling that
\[\bdp f_\epsilon = \rho_\epsilon + u_\epsilon \sdot v +  \theta_\epsilon \tfrac{|v|^2-3}{2},~~ \nabla_v n_\epsilon =0, ~~n_\epsilon =  \epsilon \divg E_\epsilon,   \]
then we can deduce that
\begin{aligno}
\label{estonlyxD1r-4}
D_{r12}&= \sum\limits_{ i \ge 1 \atop i+j=k} \intps   \nabla_x^i\left(    \tfrac{\epsilon E_\epsilon +   v\times B_\epsilon}{  \epsilon}\right) \sdot \nabla_v  \nabla_x^j ( \m n_\epsilon) \nabla_x^k \bdp f_\epsilon     \bd v \bd x \\
& = -   \sum\limits_{ i \ge 1 \atop i+j=k} \intps \left( v \sdot  \nabla_x^i E_\epsilon \right) \sdot   \nabla_x^j  n_\epsilon \nabla_x^k \bdp f_\epsilon     \bdv \bd x \\
& = - \sum\limits_{ i \ge 1 \atop i+j=k} \intps \left(    \nabla_x^i E_\epsilon  \sdot    \nabla_x^k u_\epsilon \right) \nabla_x^j n_\epsilon        \bd x \\
& = - \sum\limits_{ k-1 \ge i \ge 1 \atop i+j=k} \intps \left(    \nabla_x^i E_\epsilon  \sdot    \nabla_x^k u_\epsilon \right) \nabla_x^j n_\epsilon        \bd x \\
& -  \epsilon \intps \left(    \nabla_x^k E_\epsilon  \sdot    \nabla_x^k u_\epsilon \right)   \divg E_\epsilon        \bd x\\
& \lesssim \hd.
\end{aligno}
From \eqref{estonlyxD1r-1},  \eqref{estonlyxD1r-3}, \eqref{estonlyxD1r-4} and \eqref{estonlyxD2r}, it follows
\begin{aligno}
\label{estonlyxD1r}
D_{1r} \lesssim \hd.
\end{aligno}
and
\begin{aligno}
\label{estonlyxDr}
D_{r} \lesssim \hd.
\end{aligno}
Finally, combining \eqref{estonlyxD2} and \eqref{estonlyxDr},  for the $D_1$ in \eqref{estonlyD1-0}, it follows that
\begin{align}
\label{estonlyxD1}
D_1 \lesssim \hd.
\end{align}
For the collision term ($D_n$ in \eqref{estx-0}),  by the assumption,
\begin{aligno}
\label{estonlyxDn}
D_n & =  \reps \intps \left( \nabla_x^k   \Gamma(f_\epsilon,f_\epsilon)  \sdot \nabla_x^k f_\epsilon^\perp  +  \nabla_x^k \Gamma( h_\epsilon,f_\epsilon)  \sdot \nabla_x^k  h_\epsilon^\perp \right)   \bdv \bd x \\
& \le C\|(f_\epsilon,  h_\epsilon)\|_{H_x^s}\|(f_\epsilon,  h_\epsilon)\|_{H_{\Lambda_x}^s}\|\nabla^k_x (\reps f_\epsilon^\perp, \reps h_\epsilon^\perp)\|_{L_{\Lambda}^2} \\
& \lesssim \hd.
\end{aligno}
With the help of \eqref{estx-0},   \eqref{estonlyxD1} and \eqref{estonlyxDn}, we complete the proof.

\end{proof}
\subsection{the dissipative energy estimates of $(f_\epsilon, h_\epsilon)$ }
Denoting
\[ {H}_{v,\epsilon}^{m}(t) =     \sum\limits_{1 \le k \le m}\left(\sum\limits_{ i \ge 1, j \ge 1 \atop i + j =k}8^j\|( \nabla_v^i \nabla_x^j f_\epsilon,  \nabla_v^i \nabla_x^j  h_\epsilon)\|_{L^2}^2  + \|( \nabla_v^k   f_\epsilon,  \nabla_v^k  h_\epsilon)\|_{L^2}^2\right),~~m \ge 1,~~  \]
the following lemma is to bound the derivative of $f_\epsilon$ and $g_\epsilon$ with respective to $v$.
\begin{lemma}
\label{lemmav}
Under the assumptions  of Theorem \ref{theoremexi}, if $(f_\epsilon,  h_\epsilon, B_\epsilon, E_\epsilon)$ are solutions to \eqref{vmbtwolinear}, then
\begin{aligno}
\label{estlemmav}
\epsilon^2 \dt    \sum\limits_{m=1}^{s} \tfrac{8c_1^{s-m}}{3}{H}_{v,\epsilon}^{m}(t)    +  \tfrac{3}{4} \|(\nabla_v f_\epsilon, \nabla_v  h_\epsilon)\|_{H^{s-1}_\Lambda}^2 \lesssim \|(f_\epsilon,  h_\epsilon)\|_{H^{s}_x}^2  + \|  \epsilon E_\epsilon \|_{H^{s-1}_x}^2   + \hd.
\end{aligno}
where $c_1$ comes from the computation and is only dependent of the Sobolev embedding constants.
\end{lemma}
\begin{proof}
Applying $\nabla_x^j\nabla_v^i$ to equation of $f_\epsilon$ and $g_\epsilon$ in \eqref{vmbtwo-rewrite}, based on   the resulting equations,  we can infer that
\begin{aligno}
\label{estv-0}
& \epsilon^2 \dt \|(\nabla_x^j\nabla_v^i f_\epsilon, \nabla_x^j\nabla_v^i  h_\epsilon)\|_{L^2}^2 + \intps \left( \nabla_x^j \nabla_v^i  \mathcal{L}(f_\epsilon) \sdot \nabla_x^j\nabla_v^i f_\epsilon  +  \nabla_x^j \nabla_v^i  \mathsf{L}( h_\epsilon) \sdot \nabla_x^j\nabla_v^i  h_\epsilon \right) \bdv \bd x\\
& = - \epsilon \intps \left(  \nabla_x^j \nabla_v^i ( \vdot f_\epsilon) \sdot \nabla_x^j\nabla_v^i f_\epsilon  + \nabla_x^j \nabla_v^i ( \vdot  h_\epsilon) \sdot \nabla_x^j\nabla_v^i  h_\epsilon + \nabla_x^j \nabla_v^i (v \sdot E_\epsilon) \sdot \nabla_x^j \nabla_v^i  h_\epsilon \right) \bdv \bd x \\
& \quad
 + \epsilon^2 \intps \left(\nabla_x^j\nabla_v^i N_1 \sdot \nabla_x^j \nabla_v^i f_\epsilon   + \nabla_x^j\nabla_v^i N_2 \sdot \nabla_x^j \nabla_v^i  h_\epsilon \right) \bdv \bd x \\
& = T_1   + T_2.
\end{aligno}
Based on the assumptions in Sec. \ref{sec-assump-on-L},  we can infer that
\begin{aligno}
\label{estvleft}
& \tfrac{15}{16}\|(\nabla_x^j \nabla_v^i f_\epsilon, \nabla_x^j \nabla_v^i  h_\epsilon)\|_{L^2_\Lambda}^2  -  C \|(   f_\epsilon,    h_\epsilon)\|_{H^{k-1}}^2  \\
 &  \le \intps \left( \nabla_x^j \nabla_v^i  \mathcal{L}(f_\epsilon) \sdot \nabla_x^j\nabla_v^i f_\epsilon  +  \nabla_x^j \nabla_v^i  \mathsf{L}( h_\epsilon) \sdot \nabla_x^j\nabla_v^i  h_\epsilon \right) \bdv \bd x.
\end{aligno}
By  the H\"older inequality, we can infer
\begin{aligno}
\label{estvt1}
T_1  & \le  \epsilon \|\nabla_v^{i-1} \nabla_x^{j+1} f_\epsilon\|_{L^2}\|\nabla_v^{i} \nabla_x^{j} f_\epsilon\|_{L^2} +  \epsilon \|\nabla_v^{i-1} \nabla_x^{j+1}  h_\epsilon\|_{L^2}\|\nabla_v^{i} \nabla_x^{j}  h_\epsilon\|_{L^2} \\
& \le 4 \|\nabla_v^{i-1} \nabla_x^{j+1} (f_\epsilon,  h_\epsilon)\|_{L^2}^2 + 4 \epsilon^2\|\nabla_x^j \nabla_v^i (v  E_\epsilon)\|_{L^2}^2   + \tfrac{1}{8} \|\nabla_v^{i} \nabla_x^{j} (f_\epsilon,  h_\epsilon)\|_{L^2}^2.
\end{aligno}
For the second term $T_2$ in \eqref{estv-0},  recalling that
\begin{align*}
  N_1 & =    E_\epsilon\sdot v \sdot  h_\epsilon -  (   E_\epsilon + \reps v \times B_\epsilon) \sdot \nabla_v  h_\epsilon + \reps \Gamma(f_\epsilon, f_\epsilon), \\
  N_2 & =    E_\epsilon\sdot v \sdot f_\epsilon - (  E_\epsilon +  \reps v \times B_\epsilon) \sdot \nabla_v f_\epsilon + \reps \Gamma( h_\epsilon, f_\epsilon),
\end{align*}
  $T_2$ can be split into three parts
\begin{aligno}
\label{estvt2-0}
T_2 & = \epsilon \intps \nabla_v^j \nabla_v^i \left(   {\epsilon}  E_\epsilon\sdot v \sdot  h_\epsilon -  ( \epsilon E_\epsilon +   v \times B_\epsilon) \sdot \nabla_v  h_\epsilon \right) \sdot \nabla_v^j \nabla_x^i f_\epsilon \bdv \bd x \\
& + \epsilon \intps \nabla_v^j \nabla_v^i \left(   {\epsilon}  E_\epsilon\sdot v \sdot f_\epsilon -  ( \epsilon E_\epsilon +   v \times B_\epsilon) \sdot \nabla_v f_\epsilon \right) \sdot \nabla_v^j \nabla_x^i  h_\epsilon \bdv \bd x \\
& + \epsilon \intps \left( \nabla_x^j \nabla_v^i   \Gamma(f_\epsilon,f_\epsilon)  \sdot \nabla_x^j \nabla_v^i f_\epsilon  +  \nabla_x^j \nabla_v^i \Gamma( h_\epsilon,f_\epsilon)  \sdot \nabla_x^j \nabla_v^i  h_\epsilon \right)   \bdv \bd x.
\end{aligno}
 The  structure of first two terms on the right hand of  \eqref{estvt2-0} is similar to $D_1$ in \eqref{estonlyxD1}.
The third line in \eqref{estvt2-0} can be estimated  by the assumption on the the quadratic collision operator in Sec. \ref{sec-assump-on-L}. Thus,   we can infer that
\begin{align}
\label{estvt2}
T_2 \le \epsilon\|(E_\epsilon, B_\epsilon)\|_{H^s_x}\|(f_\epsilon, h_\epsilon)\|_{H^s_\Lambda}^2 +  \epsilon\|(f_\epsilon,  h_\epsilon)\|_{H^s}\|(f_\epsilon,  h_\epsilon)\|_{H^s_\Lambda}^2  \lesssim \hd.
\end{align}

Combining \eqref{estvleft}, \eqref{estvt1}  and \eqref{estvt2}, we can infer that
\begin{aligno}
\label{estv-1-0}
& \epsilon^2 \dt \|(\nabla_x^j\nabla_v^i f_\epsilon, \nabla_x^j\nabla_v^i  h_\epsilon)\|_{L^2}^2 - 4 \epsilon^2\|\nabla_v^i\nabla^j_x (v \sdot E_\epsilon)\|_{L^2}^2  \\
& +  \tfrac{3}{4} \|(\nabla_x^j\nabla_v^i f_\epsilon, \nabla_x^j\nabla_v^i  h_\epsilon)\|_{L^2_\Lambda}^2  -   4 \|\nabla_v^{i-1} \nabla_x^{j+1} (f_\epsilon,  h_\epsilon)\|_{L^2}^2  \\
& \lesssim \hd + \| (f_\epsilon,  h_\epsilon)\|_{H^{k-1}}^2.
\end{aligno}
Denoting the second line in \eqref{estv-1-0} by $T_{j,i}^k$($i+j=k$ and $i \ge 1$), noticing that
\begin{align*}
T^k_{k-1,1}&=\tfrac{3}{4} \|(\nabla_x^{k-1}\nabla_v^1 f_\epsilon, \nabla_x^{k-1}\nabla_v^1  h_\epsilon)\|_{L^2_\Lambda}^2  -   4 \| \nabla_x^k (f_\epsilon,  h_\epsilon)\|_{L^2}^2,\\
T^k_{k-2,2}&=\tfrac{3}{4} \|(\nabla_x^{k-2}\nabla_v^2 f_\epsilon, \nabla_x^{k-2}\nabla_v^2  h_\epsilon)\|_{L^2_\Lambda}^2  -   4 \| \nabla_x^{k-1}\nabla_v^1 (f_\epsilon,  h_\epsilon)\|_{L^2}^2,\\
~~~~~&\hspace{2cm}\vdots\hspace{2cm}\\
T^k_{1,k-1}&=\tfrac{3}{4} \|(\nabla_v^{k-1} \nabla_x^1 f_\epsilon,  \nabla_v^{k-1} \nabla_x^1  h_\epsilon)\|_{L^2_\Lambda}^2  -   4 \| \nabla_v^{k-2} \nabla_x^2 (f_\epsilon,  h_\epsilon)\|_{L^2}^2\\
T^k_{0,k}&=\tfrac{3}{4} \|(\nabla_v^k f_\epsilon,  \nabla_v^k  h_\epsilon)\|_{L^2_\Lambda}^2  -   4 \| \nabla_x^{1}\nabla_v^{k-1} (f_\epsilon,  h_\epsilon)\|_{L^2}^2,
\end{align*}
thus it follows that
\begin{aligno}
\label{estv-k}
& \epsilon^2 \dt \left(  \sum\limits_{ i \ge 1, j \ge 1 \atop i + j =k}8^j\|( \nabla_v^i \nabla_x^j f_\epsilon,  \nabla_v^i \nabla_x^j  h_\epsilon)\|_{L^2}^2  + \|( \nabla_v^k   f_\epsilon,  \nabla_v^k  h_\epsilon)\|_{L^2}^2 \right) +  \tfrac{3}{4} \sum\limits_{ i \ge 1  \atop i + j =k} \|( \nabla_v^i\nabla_x^j f_\epsilon, \nabla_v^i\nabla_x^j  h_\epsilon)\|_{L^2_\Lambda}^2\\
& \lesssim   \epsilon^2 \|E_\epsilon\|_{H^{k-1}_x}^2 +    \| (f_\epsilon,  h_\epsilon)\|_{H^k_x}^2 + \|(f_\epsilon,  h_\epsilon)\|_{H^{k-1}}^2 + \hd.
\end{aligno}
From \eqref{estv-k}, there exists some $c_1 \ge 1$ independent of $k$ such that
\begin{aligno}
\label{estv-k-2}
& \epsilon^2 \dt    {H}_{v,\epsilon}^{1}(t)  +  \tfrac{3}{4} \|\nabla_v (f_\epsilon,  h_\epsilon)\|_{L^2_\Lambda}^2   \lesssim \epsilon^2 \|E_\epsilon\|_{L^2}^2 +    \| (f_\epsilon,  h_\epsilon)\|_{H^1_x}^2  + \hd.\\
& \epsilon^2 \dt    {H}_{v,\epsilon}^{k}(t)  +  \tfrac{3}{4} \|\nabla_v (f_\epsilon,h_\epsilon)\|_{H^{k-1}_\Lambda}^2 -c_1\|(\nabla_v (f_\epsilon,h_\epsilon)\|_{H^{k-2}}^2  \\
& \lesssim \epsilon^2 \|E_\epsilon\|_{H^{k-1}_x}^2 +    \| (f_\epsilon,  h_\epsilon)\|_{H^{k}_x}^2  +  \hd,
\end{aligno}
where $k \ge 2$ and $\|h\|_{H^0} = \|h\|_{L^2}$.

By the similar method of deducing  \eqref{estv-k}, we can infer that
\begin{aligno}
& \epsilon^2 \dt    \sum\limits_{m=1}^{s} \tfrac{8c_1^{s-m}}{3}{H}_{v,\epsilon}^{m}(t)    +  \tfrac{3}{4} \|(\nabla_v f_\epsilon, \nabla_v  h_\epsilon)\|_{H^{s-1}_\Lambda}^2 \lesssim \|(f_\epsilon,  h_\epsilon)\|_{H^{s}_x}^2  + \|  \epsilon E_\epsilon \|_{H^{s-1}_x}^2   + \hd.
\end{aligno}

\end{proof}

\subsection{The dissipative estimates of the macroscopic parts}
Denoting
\[H_{c,\epsilon}^s:=\sum\limits_{k=1}^s \intps\left( \nabla_x \nabla_x^{k-1} f_\epsilon \sdot \nabla_v \nabla_x^{k-1} f_\epsilon + \nabla_x \nabla_x^{k-1}  h_\epsilon \sdot \nabla_v \nabla_x^{k-1}  h_\epsilon  \right) \bdv \bd x,    \]
the following lemma is to provide the dissipative energy estimates of $f_\epsilon$, $g_\epsilon$ with derivative to $x$.
\begin{lemma}
\label{lemmamacrox}
Under the assumptions  of Theorem \ref{theoremexi}, if $(f_\epsilon,  h_\epsilon, B_\epsilon, E_\epsilon)$ are solutions to \eqref{vmbtwolinear}, then
\begin{aligno}
\label{estlemmax}
& \epsilon \tdt H^s_{c,\epsilon} + \|\nabla(f_\epsilon, h_\epsilon)\|_{H^{s-1}_x}^2 - \delta_1\|\nabla_v(f_\epsilon, h_\epsilon)\|_{H^{s-1}_{\Lambda_x}}^2 - \delta_2 \|  E_\epsilon\|_{H^{s-1}_x}^2 \\
&  + \reps \|  n_\epsilon\|_{H^{s-1}_x}^2   \lesssim \tfrac{1}{\delta} \|( f_\epsilon^\perp, h_\epsilon^\perp)\|_{H^s_{\Lambda_x}}^2 + \hd.
\end{aligno}
\end{lemma}
\begin{proof}
According to the definition of $H^s_{c,\epsilon}$,  we can infer that
\begin{aligno}
\label{estmacrox-0}
& \epsilon   \tdt  H_{c,\epsilon}^s + \sum\limits_{k=1}^s\intps\left( \nabla_v \nabla_x^{k-1} \left( \vdot f_\epsilon\right) \sdot \nabla_x \nabla_x^{k-1} f_\epsilon + \nabla_x \nabla_x^{k-1} \left( \vdot f_\epsilon\right) \sdot \nabla_v \nabla_x^{k-1} f_\epsilon  \right) \bdv \bd x \\
& + \sum\limits_{k=1}^s\intps\left( \nabla_v \nabla_x^{k-1} \left( \vdot h_\epsilon\right) \sdot \nabla_x \nabla_x^{k-1} h_\epsilon + \nabla_x \nabla_x^{k-1} \left( \vdot h_\epsilon\right) \sdot \nabla_v \nabla_x^{k-1} h_\epsilon  \right) \bdv \bd x  \\
& - \reps \sum\limits_{k=1}^s \intps\left( \nabla_x \nabla_x^{k-1} \mathcal{L}(f_\epsilon) \sdot \nabla_v \nabla_x^{k-1} f_\epsilon + \nabla_x \nabla_x^{k-1} \mathsf{L}(h_\epsilon) \sdot \nabla_v \nabla_x^{k-1} h_\epsilon  \right) \bdv \bd x \\
& - \reps  \sum\limits_{k=1}^s\intps\left( \nabla_v \nabla_x^{k-1} \mathcal{L}( f_\epsilon) \sdot \nabla_x \nabla_x^{k-1}  f_\epsilon + \nabla_v \nabla_x^{k-1} \mathsf{L}( h_\epsilon) \sdot \nabla_x \nabla_x^{k-1}  h_\epsilon  \right) \bdv \bd x \\
& =  \sum\limits_{k=1}^s\intps \left( \nabla_v \nabla_x^{k-1}(v\!\cdot\!E_\epsilon) \sdot \nabla_x \nabla_x^{k-1}  h_\epsilon + \nabla_x \nabla_x^{k-1}(v\!\cdot\!E_\epsilon) \sdot \nabla_v \nabla_x^{k-1}  h_\epsilon \bdv \right)\bdv \bd x  \\
& +  \sum\limits_{k=1}^s\epsilon \intps \left( \nabla_x \nabla_x^{k-1}N_1 \sdot \nabla_v \nabla_x^{k-1} f_\epsilon +  \nabla_v \nabla_x^{k-1}N_1 \sdot \nabla_x \nabla_x^{k-1}  h_\epsilon \right)  \bdv \bd x\\
& + \sum\limits_{k=1}^s\epsilon \intps \left( \nabla_x \nabla_x^{k-1}N_2 \sdot \nabla_v \nabla_x^{k-1} f_\epsilon +  \nabla_v \nabla_x^{k-1}N_2 \sdot \nabla_x \nabla_x^{k-1}  h_\epsilon \right)  \bdv \bd x\\
&=M_1+ M_2+ M_3.
\end{aligno}
In what follows, we try to estimate each term in \eqref{estmacrox-0}  for each $k$ first and then sum them up. For the second term in the first line of \eqref{estmacrox-0}, denoting
\[ I : = \intps  \nabla_x \nabla_x^{k-1} \left( \vdot f_\epsilon\right) \sdot \nabla_v \nabla_x^{k-1} f_\epsilon    \bdv \bd x,   \]
by integration by parts (three times), we can infer that
\begin{align}
I &  = \intps  (v^j\sdot\nabla_{x_j x_i}^2 \nabla_x^{k-1} f_\epsilon)\partial_{v_i} \nabla_x^{k-1} f_\epsilon \bdv \bd x  = \|\nabla_x \nabla_x^{k-1}f_\epsilon\|_{L^2}^2 - \|v \sdot \nabla_x \nabla^{k-1}_x f_\epsilon\|_{L^2} - I.
\end{align}
On the other hand, we can infer that
\[\intps \nabla_v \nabla_x^{k-1}  ( \vdot f_\epsilon)  \sdot \nabla_x \nabla_x^{k-1} f_\epsilon   \bdv \bd x    =  \intps \nabla_x \nabla_x^{k-1}  ( \vdot f_\epsilon)  \sdot \nabla_v \nabla_x^{k-1} f_\epsilon   \bdv \bd x  + \|v \sdot \nabla_x \nabla^{k-1}_x f_\epsilon\|_{L^2}.  \]
All together,  it follows that
\begin{aligno}
\label{estmacroxd}
 & \intps\left( \nabla_v \nabla_x^{k-1} \left( \vdot f_\epsilon\right) \sdot \nabla_x \nabla_x^{k-1} f_\epsilon + \nabla_x \nabla_x^{k-1} \left( \vdot f_\epsilon\right) \sdot \nabla_v \nabla_x^{k-1} f_\epsilon  \right) \bdv \bd x  = \| \nabla_x \nabla^{k-1}_x f_\epsilon\|_{L^2}.
\end{aligno}
For the terms in the third and forth line of \eqref{estmacrox-0}, there exists coefficent $\reps$. The ideal is to use the microscopic part( see \eqref{estlemmaonlyx}) to deal with this difficulty.  Indeed, for the second line,  noticing that
\[ \nabla_x \nabla_x^{k-1} \mathcal{L}(f_\epsilon) =  \mathcal{L}( \nabla_x \nabla_x^{k-1} f_\epsilon)   = \mathcal{L}( \nabla_x \nabla_x^{k-1} f_\epsilon^\perp),    \]
by H\"older's inequality, we can infer that
\begin{aligno}
\label{estmacrom1}
& \vert \reps \intps\left( \nabla_x \nabla_x^{k-1} \mathcal{L}(f_\epsilon) \sdot \nabla_v \nabla_x^{k-1} f_\epsilon + \nabla_x \nabla_x^{k-1} \mathsf{L}(h_\epsilon) \sdot \nabla_v \nabla_x^{k-1} h_\epsilon  \right) \bdv \bd x   \vert \\
&  \le \tfrac{1}{2 \delta \epsilon^2} \|\nabla_x^k( f_\epsilon^\perp, h_\epsilon^\perp)\|_{L^2_\Lambda}^2 + \tfrac{\delta}{2}\|\nabla_v \nabla_x^{k-1} (f_\epsilon, g_\epsilon)\|_{L^2_\Lambda}^2,
\end{aligno}
where $\delta$ is a positive constant to be chosen later.

The third line in \eqref{estmacrox-0} is more complicated.  Decomposing $f_\epsilon$ into macroscopic part and microscopic part, we can infer
\begin{align*}
 \intps  \nabla_v \nabla_x^{k-1} \mathcal{L}(f_\epsilon) \sdot \nabla_x \nabla_x^{k-1} f_\epsilon   \bdv \bd x
& =     \intps  \nabla_v \nabla_x^{k-1} \mathcal{L}(f_\epsilon) \sdot \nabla_x \nabla_x^{k-1} \bdp f_\epsilon   \bdv \bd x   \\
& +    \intps  \nabla_v \nabla_x^{k-1} \mathcal{L}( f_\epsilon  ) \sdot \nabla_x \nabla_x^{k-1}   f_\epsilon^\perp   \bdv \bd x .
\end{align*}
After integration by parts two times, we can infer that
\begin{align*}
\intps  \nabla_v \nabla_x^{k-1} \mathcal{L}(f_\epsilon) \sdot \nabla_x \nabla_x^{k-1} \bdp f_\epsilon   \bdv \bd x &  = \intps  \nabla_x \nabla_x^{k-1} \mathcal{L}(f_\epsilon) \sdot \nabla_v\nabla_x^{k-1} \bdp f_\epsilon   \bdv \bd x \\
& - \intps  \nabla_x \nabla_x^{k-1} \mathcal{L}(f_\epsilon) \sdot  v\nabla_x^{k-1} \bdp f_\epsilon   \bdv \bd x .
\end{align*}
Then similar to \eqref{estmacrom1}, we can infer
\begin{aligno}
\label{estmacrom2}
& \vert \reps \intps\left( \nabla_v \nabla_x^{k-1} \mathcal{L}( f_\epsilon) \sdot \nabla_x \nabla_x^{k-1}  f_\epsilon + \nabla_v \nabla_x^{k-1} \mathsf{L}( f_\epsilon) \sdot \nabla_x \nabla_x^{k-1}  f_\epsilon  \right) \bdv \bd x \vert    \\
&  \le \tfrac{1}{ 2\delta_1 \epsilon^2} \|\nabla_x^k( f_\epsilon^\perp, h_\epsilon^\perp)\|_{L^2_\Lambda}^2 +  \tfrac{\delta_1}{2} \|\nabla_v \nabla_x^{k-1} (f_\epsilon, g_\epsilon)\|_{L^2_\Lambda}^2,
\end{aligno}
For $M_1$  in the right hand of \eqref{estmacrox-0}, by integration by parts, it follows that
\begin{aligno}
\label{estmacroe-1}
M_1& =   \intps \left( -2\nabla_v \nabla_x^{k-1}(v\!\cdot\!E_\epsilon) \sdot \nabla_x \nabla_x^{k-1}  h_\epsilon +   \nabla_x^{k-1}(v\!\cdot\!E_\epsilon)  (v \sdot \nabla_x \nabla_x^{k-1}  h_\epsilon)  \bdv \right)\bdv \bd x \\
& =    \intps \left(  -2 \nabla_x^{k-1} E_\epsilon \sdot \nabla_x \nabla_x^{k-1}  h_\epsilon + \nabla_x^{k-1}(v\!\cdot\!E_\epsilon) \sdot v \sdot \nabla_x \nabla_x^{k-1}  (n_\epsilon + h_\epsilon^\perp) \bdv \right)\bdv \bd x.
\end{aligno}
Recalling that
\[\divg E_\epsilon = \reps n_\epsilon,  \]
and by simple computation, we can infer that
\begin{align}
\label{estmacroe-2}
\intps \left(  -2 \nabla_x^{k-1} E_\epsilon \sdot \nabla_x \nabla_x^{k-1}  h_\epsilon + \nabla_x^{k-1}(v\!\cdot\!E_\epsilon) \sdot v \sdot \nabla_x \nabla_x^{k-1}   n_\epsilon  \bdv \right)\bdv \bd x = \reps \|\nabla_x^{k-1} n_\epsilon\|_{L^2}^2.
\end{align}
Combining \eqref{estmacroe-1} and \eqref{estmacroe-2} and by H\"older's inequality, we can infer
\begin{align}
\label{estmacroxm1}
M_1 + \reps \|  n_\epsilon\|_{H^{s-1}}^2 - \delta_2 \|  E_\epsilon\|_{H^{s-1}}^2 \lesssim \tfrac{1}{\delta} \|\nabla_x^{k-1}h_\epsilon^\perp\|_{L^2}^2.
\end{align}
For the left two terms in the right hand of \eqref{estmacrox-0}, by   he same trick to that of deducing \eqref{estonlyxD2} and \eqref{estmacroxd}, we can infer that
\begin{align}
\label{estmacrom23}
M_2 + M_3 \lesssim \hd.
\end{align}

In the light of \eqref{estmacrox-0}, \eqref{estmacroxd},  \eqref{estmacrom2}, \eqref{estmacrom1} and \eqref{estmacrom23}, we can infer that
\begin{aligno}
\label{estmacrox-1}
& \epsilon \tdt H^s_{c,\epsilon} + \|\nabla(f_\epsilon, h_\epsilon)\|_{H^{s-1}_x}^2 - \delta_1\|\nabla_v(f_\epsilon, h_\epsilon)\|_{H^{s-1}_{\Lambda_x}}^2 - \delta_2 \|  E_\epsilon\|_{H^{s-1}_x}^2 \\
&  + \reps \|  n_\epsilon\|_{H^{s-1}_x}^2   \lesssim \tfrac{1}{\delta} \|( f_\epsilon^\perp, h_\epsilon^\perp)\|_{H^s_{\Lambda_x}}^2 + \hd.
\end{aligno}

\end{proof}

\subsection{The dissipative estimates of the electromagnetic parts}
Denoting
\begin{aligno}
\label{normeletro}
H^{s}_{\epsilon,e}(t) = \tdt \left[  \epsilon^2 \sum\limits_{k=0}^{s-1} \intt      \nabla_x^k \tilde j_\epsilon \sdot    \nabla_x^k E_\epsilon  \bd x +   \sum\limits_{k=0}^{s-2} \intt     \left( \epsilon^2 \curl \nabla_x^k \tilde j_\epsilon \sdot  \curl  \nabla_x^k E_\epsilon   +  \epsilon \sdot \tfrac{3\sigma}{4} \nabla_x^k E_\epsilon\sdot \curl \nabla_x^k B_\epsilon \right)\bd x  \right],
\end{aligno}
the following lemma provides the dissipative energy estimates of $B_\epsilon$ and $E_\epsilon$.
\begin{lemma}
\label{lemma-curl-e}
Under the assumptions  of Theorem \ref{theoremexi}, if $(f_\epsilon,  h_\epsilon, B_\epsilon, E_\epsilon)$ are solutions to \eqref{vmbtwolinear}, then
\begin{aligno}
\label{estdissipative-ele}
-\tdt H^s_{\epsilon,e} +   \tfrac{3 \sigma}{4} \sum\limits_{k=0}^{s-1} \| \nabla_x^k E_\epsilon\|_{L^2}^2 +    \tfrac{  \sigma}{4}\sum\limits_{k=1}^{s-2}  \|\curl \nabla_x^{k-1} B_\epsilon\|_{L^2}^2 \lesssim \repst \| h_\epsilon^\perp\|_{H^s_x}^2 + \hd.
\end{aligno}
 \end{lemma}

\begin{proof}
Recalling that
  \[  \tilde{j}_\epsilon = \reps \intv h_\epsilon \tilde{v},~~\sigma = \tfrac{1}{3} \intv \tilde{v}\sdot v \bdv,~~\mathsf{L}(\tilde{v}) = v, \]
from the second equation of \eqref{vmbtwo-rewrite}, we can infer that $\tilde{j}_\epsilon$ satisfies the following equation:
\begin{align}
\label{equationJtilde}
\epsilon^2 \partial_t \tilde{j}_\epsilon\ +   \divg  \intv \tilde{v}\otimes v  h_\epsilon  \bdv  -  \sigma   E_\epsilon - j_\epsilon  =  \epsilon \intv \tilde v  N_1 \bdv.
\end{align}
There exists a ``damping'' term of $E_\epsilon$ in equation \eqref{equationJtilde}. On the other hand, the equations of $E_\epsilon$ and $B_\epsilon$ are
\begin{aligno}
\label{equationbe}
&\epsilon\partial_t E_\epsilon - \curl B_\epsilon = -   j_\epsilon, \\
&\partial_t B_\epsilon + \curl E_\epsilon =0.
\end{aligno}
The dissipative energy estimates of $E_\epsilon$ can be deduced by employing structrue of \eqref{equationJtilde} and \eqref{equationbe}. Indeed,  after applying $\curl \nabla_x^k$ to \eqref{equationJtilde} and \eqref{equationbe}, we can infer
\begin{aligno}
\label{est-curl-e-1}
& - \epsilon^2\intt \partial_t \curl \nabla_x^k \tilde j_\epsilon \sdot  \curl \nabla_x^k E_\epsilon \bd x -    \intt  \nabla\timess\left(\divg  \intv \tilde v \otimes v \nabla_x^k  h_\epsilon  \bdv\right) \sdot \curl \nabla_x^k E_\epsilon \bd x \\
& +  \sigma   \|\curl \nabla_x^k E_\epsilon\|_{L^2}^2 -  \intt  \curl \nabla_x^k {j}_\epsilon \sdot \curl \nabla_x^k E_\epsilon \bd x = \epsilon \intt \curl\intv \tilde v \nabla_x^k N_2 \bdv \sdot \curl \nabla_x^k E_\epsilon \bd x,
\end{aligno}
and
\begin{aligno}
\label{est-curl-e-2}
& - \epsilon^2 \intt  \curl \nabla_x^k \tilde j_\epsilon \sdot \partial_t \curl \nabla_x^k E_\epsilon \bd x +  \epsilon\intt \curl \curl \nabla_x^k B_\epsilon \sdot \curl \nabla_x^k \tilde j_\epsilon \bd x  - \epsilon \intt \curl \nabla_x^k j_\epsilon \sdot \curl \nabla_x^k \tilde{j}_\epsilon \bd x =0.
\end{aligno}
With the help of  \eqref{est-curl-e-1} and \eqref{est-curl-e-2}, we can infer that
 \begin{align}
\label{est-curl-3}
\begin{split}
&     -\epsilon^2 \tdt \intt     \sdot \curl \nabla_x^k \tilde j_\epsilon \sdot  \curl  \nabla_x^k E_\epsilon  \bd x  +  \sigma  \|\curl \nabla_x^k E_\epsilon\|_{L^2}^2 \\
 & =    \intt  \nabla\timess\left(\divg  \intv \tilde v\otimes v \nabla_x^k  h_\epsilon  \bdv\right) \sdot \curl \nabla_x^k E_\epsilon \bd x \\
 &  +  \epsilon \intt \curl \nabla_x^k j_\epsilon \sdot \curl \nabla_x^k \tilde{j}_\epsilon \bd x  -     \epsilon\intt \curl \curl \nabla_x^k B_\epsilon \sdot \curl \nabla_x^k \tilde j_\epsilon \bd x\\
 & - \epsilon \intt \curl\intv \tilde v \nabla_x^k N_2 \bdv \sdot \curl \nabla_x^k E_\epsilon \bd x + \intt  \curl \nabla_x^k {j}_\epsilon \sdot \curl \nabla_x^k E_\epsilon \bd x.
\end{split}
\end{align}
Noticing that
\[  \divg  \intv  \tilde{v}\otimes v  \nabla_x^k  h_\epsilon  \bdv = \divg  \intv  \tilde{v}\otimes v  \nabla_x^k  h_\epsilon^\perp  \bdv + \sigma \nabla n_\epsilon,  \]
then by H\"older's inequality, we can infer that
\begin{align*}
  \big\vert   \intt  \nabla\timess\left(\divg  \intv  \tilde{v}\otimes v  \nabla_x^k  h_\epsilon  \bdv\right) \sdot \curl \nabla_x^k  E_\epsilon \bd x\big\vert   \le C \|\nabla^{k+2}_x  h_\epsilon^\perp\|_{L^2}^2 + \tfrac{\sigma}{16} \|\curl \nabla_x^k E_\epsilon\|_{L^2}^2.
\end{align*}
Recalling that
\[  j_\epsilon = \reps \intv h_\epsilon^\perp v \bdv,~~ \tilde{j}_\epsilon = \intv \tilde{v}h_\epsilon^\perp\bdv,    \]
we can infer that
\begin{aligno}
\epsilon \intt \curl \nabla_x^k j_\epsilon \sdot \curl \nabla_x^k \tilde{j}_\epsilon \bd x - \epsilon\intt \curl \curl \nabla_x^k B_\epsilon \sdot \curl \nabla_x^k \tilde j_\epsilon \bd x + \intt  \curl \nabla_x^k {j}_\epsilon \sdot \curl \nabla_x^k E_\epsilon \bd x\\
\le C \repst \|\nabla_x^{k+1} h_\epsilon^\perp\|_{H^1_x}^2 + \tfrac{\sigma}{16}\|\curl \nabla_x^k (B_\epsilon,E_\epsilon)\|_{L^2}^2.
\end{aligno}

For the left term on the right hand of \eqref{est-curl-3}, noticing that $k \le s-2$, it follows that
\begin{aligno}
\label{esteng0}
& \epsilon \intt \curl\intv \tilde v \nabla_x^k N_2 \bdv \sdot \curl \nabla_x^k E_\epsilon \bd x \\
& = \epsilon \sdot \epsilon \intt\curl \intv \tilde{v}\nabla_x^k\left(   E_\epsilon\sdot v \sdot f_\epsilon   \right) \bdv \sdot   \curl\nabla_x^k E_\epsilon \bd x\\
& -  \intt\curl \intv \tilde{v}\nabla_x^k\left(    ( \epsilon E_\epsilon +    v \times B_\epsilon) \sdot \nabla_v f_\epsilon + \Gamma( h_\epsilon, f_\epsilon) \right)\bdv \sdot     \curl\nabla_x^k E_\epsilon \bd x \\
& \le C \hd +  \tfrac{\sigma}{16}\tfrac{\epsilon^2}{\epsilon^2}\|\curl \nabla_x^k E_\epsilon\|_{L^2}^2.
\end{aligno}
All together, we can infer that
\begin{aligno}
\label{estlemmade1}
&     -\epsilon^2 \tdt \intt     \sdot \curl \nabla_x^k \tilde j_\epsilon \sdot  \curl  \nabla_x^k E_\epsilon  \bd x  +   \tfrac{3 \sigma}{4}  \|\curl \nabla_x^k E_\epsilon\|_{L^2}^2 - \tfrac{  \sigma}{16}  \|\curl \nabla_x^k B_\epsilon\|_{L^2}^2 \\
& \lesssim \repst \|\nabla_x^{k+1} h_\epsilon^\perp\|_{H^1_x}^2 + \hd,
\end{aligno}
and
\begin{aligno}
\label{estlemmade2}
&     -\epsilon^2 \tdt \sum\limits_{k=0}^{s-2} \intt     \sdot \curl \nabla_x^k \tilde j_\epsilon \sdot  \curl  \nabla_x^k E_\epsilon  \bd x  +   \tfrac{3 \sigma}{4} \sum\limits_{k=0}^{s-2} \|\curl \nabla_x^k E_\epsilon\|_{L^2}^2 - \tfrac{  \sigma}{16}\sum\limits_{k=0}^{s-2}  \|\curl \nabla_x^k B_\epsilon\|_{L^2}^2 \\
& \lesssim \repst \| h_\epsilon^\perp\|_{H^s_x}^2 + \hd.
\end{aligno}
By the similar way of deducing \eqref{est-curl-3} and \eqref{estlemmade2}, we can infer that
\begin{aligno}
\label{estlemmade3}
&     -\epsilon^2 \tdt \sum\limits_{k=0}^{s-1} \intt      \nabla_x^k \tilde j_\epsilon \sdot    \nabla_x^k E_\epsilon  \bd x  +   \tfrac{3 \sigma}{4} \sum\limits_{k=0}^{s-1} \| \nabla_x^k E_\epsilon\|_{L^2}^2 - \tfrac{  \sigma}{8}\sum\limits_{k=1}^{s-1}  \|\curl \nabla_x^{k-1} B_\epsilon\|_{L^2}^2 \\
& \lesssim \repst \| h_\epsilon^\perp\|_{H^s_x}^2 + \hd.
\end{aligno}

Combing \eqref{estlemmade2} and \eqref{estlemmade3}, we can infer that
\begin{aligno}
\label{estlemmade}
&     - \epsilon^2 \sdot \tdt \left(    \sum\limits_{k=0}^{s-1} \intt      \nabla_x^k \tilde j_\epsilon \sdot    \nabla_x^k E_\epsilon  \bd x +  \sum\limits_{k=0}^{s-2} \intt     \sdot \curl \nabla_x^k \tilde j_\epsilon \sdot  \curl  \nabla_x^k E_\epsilon  \bd x \right) \\
&  +   \tfrac{3 \sigma}{4} \sum\limits_{k=0}^{s-1} \| \nabla_x^k E_\epsilon\|_{L^2}^2 + \tfrac{3 \sigma}{4} \sum\limits_{k=0}^{s-2} \|\curl \nabla_x^k E_\epsilon\|_{L^2}^2 - \tfrac{  \sigma}{4}\sum\limits_{k=1}^{s-2}  \|\curl \nabla_x^{k-1} B_\epsilon\|_{L^2}^2 \\
& \lesssim \repst \| h_\epsilon^\perp\|_{H^s_x}^2 + \hd.
\end{aligno}
To finish the proof of this lemma, we still need to obtain the dissipative energy estimates of magnetic field. From the equations of $B_\epsilon$ and $E_\epsilon$,    we can infer that
\begin{align*}
- \epsilon     \sdot \frac{\bd }{\bd t} \intt \nabla_x^k E_\epsilon\sdot \curl \nabla_x^k B_\epsilon\bd x +   \|\curl \nabla_x^k B_\epsilon\|_{L^2}^2 - \epsilon \|\curl \nabla_x^k E_\epsilon\|_{L^2}^2 =   \intt \nabla_x^k j_\epsilon \sdot \curl \nabla_x^k B_\epsilon \bd x.
\end{align*}
By the H\"older's inequality, we obtain that
\begin{align}
\label{estlemmab}
- \epsilon     \sdot \frac{\bd }{\bd t} \intt \nabla_x^k E_\epsilon\sdot \curl \nabla_x^k B_\epsilon\bd x +   \tfrac{3}{4}\|\curl \nabla_x^k B_\epsilon\|_{L^2}^2 - \epsilon \|\curl \nabla_x^k E_\epsilon\|_{L^2}^2 \lesssim \repst \|h^\perp_\epsilon\|_{H^s_x}^2.
\end{align}
Combining \eqref{estlemmade} and \eqref{estlemmab}, we complete the proof.

\end{proof}

\subsection{The whole estimates}
\label{sec-whole}
In the left of this section, based on \eqref{estlemmaonlyx}, \eqref{estlemmav}, \eqref{estlemmax}, \eqref{estlemmade} and \eqref{estlemmab}, we can obtain estimates like this
\[  \tdt \mathcal{H}_\epsilon^s  + \mathcal{D}_\epsilon^s \lesssim \hd.  \]
Then, if the initial data is some enough, we can obtain uniform estimates of solutions with respect to the Knudsen number.
\begin{lemma}
\label{lemmawhole}
Under the assumptions  of Theorem \ref{theoremexi}, if $(f_\epsilon,  h_\epsilon, B_\epsilon, E_\epsilon)$ are solutions to \eqref{vmbtwolinear}, then there exists some small enough constant $c_0$ such that
\begin{align}
\label{estlemmawhole}
\sup\limits_{ 0 \le s \le t}  {H}_\epsilon^s(t) +  \tfrac{1}{4} \int_0^t \left(   \|(   f_\epsilon,     h_\epsilon)\|_{H^{s}_\Lambda}^2 +  \tfrac{\epsilon^2}{\epsilon^2} \|(E_\epsilon, B_\epsilon)\|_{H^{s-1}_{ x}}^2 +  \|(   f_\epsilon^\perp,     h_\epsilon^\perp)\|_{H^{s}_{\Lambda_x}}^2 \right)(s)\bd s \le \tfrac {c_u}{c_l}  {H}_\epsilon^s(0),
\end{align}
where $c_l$ and $c_u$ are positive constants only dependent of the Sobolev embedding constant.
\end{lemma}
\begin{proof}This Lemma can be proved by employing the Poincare's inequality and choosing proper consants.  Combining \eqref{estmacrox-1} and \eqref{estdissipative-ele} up and setting the $\delta_2=\tfrac{\sigma}{4}$ (in \eqref{estmacrox-1}), then  we can infer that
\begin{aligno}
\label{estwhole-e+v}
& \tdt \bigg( H^{s}_{c,\epsilon} - H^{s}_{\epsilon,e}\bigg)(t)  + \tfrac{3}{4}\|(\nabla_x f_\epsilon, \nabla_x  h_\epsilon)\|_{H^{s-1}_x}^2   + \reps \|n_\epsilon\|_{H^{s-1}_x}^2       - \delta_1 \|\nabla_v(f_\epsilon,  h_\epsilon)\|_{H^{s-1}_{\Lambda_x}}^2 \\
& + \tfrac{ \sigma}{2} \sum\limits_{k=0}^{s-1} \| \nabla_x^k E_\epsilon\|_{L^2}^2 +    \tfrac{  \sigma}{4}\sum\limits_{k=1}^{s-2}  \|\curl \nabla_x^{k-1} B_\epsilon\|_{L^2}^2 \lesssim \repst \| (f_\epsilon
^\perp,h_\epsilon^\perp)\|_{H^s_x}^2 + \hd.
 \end{aligno}
According to the definition of $\mathcal{D}_\epsilon^s$,  the dissipative energy estimates of $f_\epsilon$, $h_\epsilon$ and $B_\epsilon$ is not complete. We need to recover the $L^2$ estimates of $f_\epsilon$, $h_\epsilon$ and $B_\epsilon$. For the magnetic field $B_\epsilon$, from \eqref{estMeaninitial},
\[ \intt B_\epsilon(t) \bd x =0,,~~ \forall t \ge 0,  \]
on the other hand,
\[ \divg B_\epsilon =0, \]
then by Poincare's inequality, we can infer that
\begin{aligno} \label{estequalb} \sum\limits_{k=1}^{s-2} \|\curl \nabla_x^{k-1} B_\epsilon\|_{L^2}^2 \approx  \|B_\epsilon\|_{H^{s-1}_x}^2.
\end{aligno}
For $f_\epsilon$ and $h_\epsilon$, recalling that we can decompose $f_\epsilon$ and $h_\epsilon$ like this
\[ f_\epsilon = \bdp f_\epsilon + f_\epsilon^\perp,~~h_\epsilon = n_\epsilon + h_\epsilon^\perp, \]
from \eqref{estMeaninitial},  we can infer that
\[ \intt \bdp f_\epsilon(t) \bd x = - \epsilon  v \sdot \intt E_\epsilon\times B_\epsilon(t) \bd x    - \epsilon \tfrac{|v|^2-3}{6} \|(\sqrt{\epsilon} E_\epsilon, B_\epsilon)(t)\|_{L^2}^2,~~~~ \intt n_\epsilon(t) \bd x =0.  \]
Based on the above mean value,  it follow that
\begin{align}
\label{estpoing}
 \|(f_\epsilon,h_\epsilon)\|_{H^s_x}^2 \lesssim  \|\nabla_x(f_\epsilon,h_\epsilon)\|_{H^{s-1}_x}^2 + \|(f_\epsilon^\perp,h_\epsilon^\perp)\|_{H^s_x}^2 + \hd.
\end{align}
Combining \eqref{estwhole-e+v}, \eqref{estequalb} and \eqref{estpoing}, there exists some $c_6>0$ such that
\begin{aligno}
\label{estwhole-miss-v}
& \tdt \bigg( H^{s}_{c,\epsilon} - H^{s}_{\epsilon,e}\bigg)(t)  + c_6(\|(  f_\epsilon,    h_\epsilon)\|_{H^{s}_x}^2 +   \|(E_\epsilon, B_\epsilon)\|_{H^{s-1}_{ x}}^2)   + \reps \|n_\epsilon\|_{H^{s-1}_x}^2        - \delta_2 \|\nabla_v(f_\epsilon,  h_\epsilon)\|_{H^{s-1}_{\Lambda_x}}^2       \\
& \lesssim \repst \| (f_\epsilon
^\perp,h_\epsilon^\perp)\|_{H^s_x}^2 + \hd.
\end{aligno}
To get the whole $\mathcal{D}_\epsilon^s$,  \eqref{estwhole-miss-v} and \eqref{estlemmav} should be put together. From \eqref{estlemmav}, there exists some $c_7$ such that
\begin{aligno}
\label{estlemmav-whole}
& \epsilon^2 \dt    \sum\limits_{m=1}^{s} \tfrac{8c_1^{s-m}}{3}{H}_{v,\epsilon}^{m}(t)    +  \tfrac{3}{4} \|(\nabla_v f_\epsilon, \nabla_v  h_\epsilon)\|_{H^{s-1}_\Lambda}^2 -c_7 \|(f_\epsilon,  h_\epsilon)\|_{H^{s}_x}^2 - c_7 \|  \epsilon E_\epsilon \|_{H^{s-1}_x}^2   \lesssim \hd.
\end{aligno}
From \eqref{estwhole-miss-v} and \eqref{estlemmav-whole}, choosing $c_8$ and $\delta_2$ such that
\[ c_8 c_6 \ge c_7 + \tfrac{1}{2}, ~~c_8 \delta_2 = \tfrac{1}{4},~~ c_8 \ge 1,  \]
then there exists some $d_1 >0$ such that
\begin{aligno}
\label{estwhole-miss-micro}
& \tdt \bigg( 2c_8 H^{s}_{c,\epsilon} - 2c_8 H^{s}_{\epsilon,e} + \epsilon^2    2 \sum\limits_{m=1}^{s} \tfrac{8c_1^{s-m}}{3}{H}_{v,\epsilon}^{m}(t)\bigg)(t)  \\
& + \|(  f_\epsilon,    h_\epsilon)\|_{H^{s}_x}^2 +   \|(E_\epsilon, B_\epsilon)\|_{H^{s-1}_{ x}}^2   + \reps \|n_\epsilon\|_{H^{s-1}_x}^2     \\
& + \|\nabla_v(f_\epsilon,  h_\epsilon)\|_{H^{s-1}_{\Lambda_x}}^2       -  \tfrac{d_1}{\epsilon^2} \| (f_\epsilon^\perp,h_\epsilon^\perp)\|_{H^s_x}^2 \lesssim \hd.
\end{aligno}
Finally, from \eqref{estlemmaonlyx}, denoting
\begin{align}
\label{normequalhs}
 \mathcal{ \tilde H}_\epsilon^s :=  2c_8 \epsilon H^{s}_{c,\epsilon} - 2c_8 H^{s}_{\epsilon,e} +  2  \epsilon^2 \sum\limits_{m=1}^{s} \tfrac{8c_1^{s-m}}{3}{H}_{v,\epsilon}^{m} + d_2 \|  (f_\epsilon,  h_\epsilon, \seps E_\epsilon, B_\epsilon)\|_{H^s_x}^2,
\end{align}
where  $d_2$ are chosen to satisfy
\[d_2 \ge d_1 +1,~~, \mathcal{ \tilde H}_\epsilon^s \approx  \mathcal{  H}_\epsilon^s,  \]
then it follow that there exists some positive constant $d_3$ such that
\begin{aligno}
\label{estwhole-1}
& \tdt \tilde{H}_\epsilon^s + \mathcal{D}_\epsilon^s  \le d_3 \mathcal{D}^s_\epsilon(t) \sqrt{\mathcal{ \tilde H}_\epsilon^s}(t).
\end{aligno}
If the initial data satisfy
\begin{align}
H_\epsilon^s(0) \le c_0:=\tfrac{1}{4 d_3^2},
\end{align}
then we can infer that for any $ t>0$
\begin{aligno}
\sup\limits_{ 0 \le s \le t} \tilde{H}_\epsilon^s(t) +  \tfrac{1}{2} \int_0^t \mathcal{D}^s_\epsilon(\tau)\bd \tau \le \tilde{H}_\epsilon^s(0).
\end{aligno}
On the other hand,  $\tilde{H}_\epsilon^s$ is equivalent to $H_\epsilon^s$, i.e., there exist $0<c_l<1$ and $c_u>0$ such that
\begin{align}
\label{estclcu}
c_l \| {H_\epsilon^s} \le \tilde{H}_\epsilon^s \le c_u  H^s_\epsilon.
\end{align} Thus, we can infer that for any $ t >0$
\begin{align*}
\sup\limits_{ 0 \le s \le t}  {H}_\epsilon^s(t) +   \tfrac{1}{2} \int_0^t \mathcal{D}^s_\epsilon(\tau)\bd \tau \le \tfrac {c_u}{c_l}  {H}_\epsilon^s(0),~~\forall \epsilon \in  (0,1].
\end{align*}
We complete the proof of this lemma.
\end{proof}
\subsection{The existence of system \ref{vmbtwolinear} }
For each fixed $\epsilon$,  the existence of solutions to system \ref{vmbtwolinear} can be found in \cite{guo-2003-vmb-invention}. But it also can be obtained by employing  the following iteration system ($n \ge 1$):
\begin{align}
\label{vmb-appro}
\begin{cases}
\partial_t f_\epsilon^n +  \reps \vdot f_\epsilon^n  -   \repst \mathcal{L}(f_\epsilon^n)  + \tfrac{\epsilon E_\epsilon^{n-1} +   v\times B_\epsilon^{n-1}}{\m \epsilon} \sdot \nabla_v (\m  h_\epsilon^n) = \reps\Gamma(f_\epsilon^{n-1},f_\epsilon^{n-1}), \\
\partial_t  h_\epsilon^n +  \reps \vdot  h_\epsilon^n   - \tfrac{\epsilon}{\epsilon^2}  E_\epsilon^n \sdot v -   \repst \mathsf{L}( h_\epsilon^n)  +  \tfrac{\epsilon E_\epsilon^{n-1}    v\times B_\epsilon^{n-1}}{\m \epsilon} \sdot \nabla_v (\m f_\epsilon^n) = \reps\Gamma( h_\epsilon^{n-1},f_\epsilon^{n-1}), \\
\partial_t E_\epsilon^n - \curl B_\epsilon^n = -  j_\epsilon^n, \\
  \partial_t B_\epsilon^n + \curl E_\epsilon^n =0,\\
\divg B_\epsilon^n =0,~~\epsilon \divg E_\epsilon^n =  \intv  h_\epsilon^n\bdv,\\
(f_\epsilon^n, h_\epsilon^n, B_\epsilon^n, E_\epsilon^n)(0)=(f_\epsilon, h_\epsilon, B_\epsilon, E_\epsilon)(0),
\end{cases}
\end{align}
with
\begin{align*}
f^0_\epsilon=g^0=0,~~E^0_\epsilon=B^0_\epsilon=0.
\end{align*}
The approximate solutions can be constructed by iteration method. Then based on the uniform estimates (obtained by induction method), the solutions can be obtained by employing Rellich-Kondrachov compactness theorem.

\section{The proof of Theorem \ref{theoremlimit}}
\label{sec-limit}
In Sec.\ref{secEstimates}, for any $t>0$, the solution$(f_\epsilon, h_\epsilon, B_\epsilon, E_\epsilon)$ satisfy the following uniform estimate:
\begin{aligno}
\label{estwhole}
&\sup\limits_{ 0 \le s \le t}  \|(f_\epsilon, h_\epsilon,B_\epsilon, \sqrt\epsilon E_\epsilon)(s)\|_{H^s_x}^2 +   \int_0^t    \|(   f_\epsilon,     h_\epsilon)(\tau)\|_{H^{s}_\Lambda}^2 \bd \tau \\
& + \int_0^t \left(     \repst \|(   f_\epsilon^\perp,     h_\epsilon^\perp)\|_{H^{s}_{\Lambda_x}}^2 + \reps \|(n_\epsilon, \sqrt\epsilon B_\epsilon, \sqrt{\epsilon}E_\epsilon)\|_{H^{s-1}}^2\right)(\tau) \bd \tau \le C_0
\end{aligno}
Based on \eqref{estwhole}, we can verify the MHD limit of VMB system.

{\bf Step 1: the limit of $f_\epsilon$ and $ h_\epsilon$.}

First, noticing that
\[  \int_0^t  \left(   \|(   f_\epsilon,     h_\epsilon)(\tau)\|_{H^{s}_\Lambda}^2 \bd \tau  +     \repst \|(   f_\epsilon^\perp,     h_\epsilon^\perp)(\tau)\|_{H^{s}_{\Lambda_x}}^2  + \reps \|n_\epsilon(\tau)\|_{H^{s-1}}^2  \right) \bd \tau \le C_0,~~\forall t >0, \]
and recalling
\[ f_\epsilon = \rho_\epsilon + u_\epsilon \sdot v + \tfrac{|v|^2-3}{2}\theta_\epsilon + f_\epsilon^\perp,~~h_\epsilon = n_\epsilon + h^\perp_\epsilon,   \]
then we can infer that there exist   $\rho, ~~u,~~\theta,~~n$ belonging to $H^s_x $  space such that
\begin{align}
\label{estlimitn}
f_\epsilon \to f=\rho(t,x) + u(t,x) \sdot v + \tfrac{|v|^2-3}{2} \theta(t,x),~~ h_\epsilon \to n= 0,~~~~\text{in},~~ L^2\left((0,+\infty); H^{s-1}_{x}\right).
\end{align}
Furthermore,   there exist   $\{  B, E\} \subset H^s_x $  such that
\begin{align}
\label{estconvergenceMarco}
\rho_\epsilon \to \rho,~~u_\epsilon \to u,~~ \theta_\epsilon \to \theta,~~n_\epsilon \to 0,~~B_\epsilon \to B,~~E_\epsilon \to E,~~\text{in},~~ L^2((0,+\infty);H^{s-1}_x).
\end{align}
The next step is to verify that $(\rho, u, \theta, B, E)$ satisfies the  MHD system.

{\bf Step 2:   the limiting equation}

From \eqref{vmbtwo-rewrite}, based on  the local conservation laws, we can find that $\rho_\epsilon, u_\epsilon, \theta_\epsilon, B_\epsilon, E_\epsilon$ satisfy the following system:
\begin{align}
\label{fluid-app}
\begin{cases}
\partial_t \rho_\epsilon + \tfrac{1}{\epsilon} \divg u_\epsilon =0,\\
\partial_t u_\epsilon + \reps \divg \intv \hat{A} \mathcal{L}f_\epsilon \bdv + \reps \nabla_x (\rho_\epsilon + \theta_\epsilon )  =   n_\epsilon \sdot E_\epsilon +   j_\epsilon \timess B_\epsilon,\\
\partial_t \theta_\epsilon + \tfrac{2}{3\epsilon} \divg \intv \hat{B} \mathcal{L}f_\epsilon \bdv + \tfrac{2}{3\epsilon} \divg u_\epsilon  =  \epsilon \sdot \tfrac{2}{3} j_\epsilon \cdot E_\epsilon,\\
\epsilon \partial_t E_\epsilon - \curl B_\epsilon = -   j_\epsilon, \\
   \partial_t B_\epsilon + \curl E_\epsilon =0,\\
\divg B_\epsilon =0,~~ \epsilon \divg E_\epsilon =   n_\epsilon.
\end{cases}
\end{align}
where
\begin{aligno}
A(v) = v \otimes v - \tfrac{|v|^2}{3}\mathbb{I},~~B(v) = v ( \tfrac{|v|^2}{2} - \tfrac{5}{2}), ~\mathcal{L}\hat{A}(v) = A(v),~~\mathcal{L}\hat{B}(v) = B(v).
\end{aligno}
By the first equation of \eqref{fluid-app},    in the distributional sense
\begin{align}
\label{estConvergenceDivgu}
\divg u_\epsilon \to \divg u  =0.
\end{align}
Furthermore, recalling that
\[ j_\epsilon = \reps \intv h_\epsilon^\perp v \bdv,~~ \tilde{j} = \reps \intv h_\epsilon^\perp \tilde{v} \bdv, \]
from \eqref{estwhole}, we can infer that
\begin{align}
\label{estlimitj}
\int_0^\infty \|  ( j_\epsilon, \tilde{j}_\epsilon)(\tau)\|_{H^s_x}^2 \bd \tau \lesssim \epsilon^2.
\end{align}
Then from the forth equation of \eqref{fluid-app}, we can infer that
\begin{align}
\label{estfluid-j}
j_\epsilon \to j=\curl B,~~\text{in},~~ L^2\left((0,+\infty); H^{s-1}_{x}\right).
\end{align}

{\bf Step 3, the limiting equation of $u$.}

Furthermore, for   velocity and temperature equation in \eqref{fluid-app}, we can infer that
\begin{aligno}
\label{nsp-app}
\partial_t u_\epsilon + \reps \divg \intv \hat{A} \mathcal{L}f_\epsilon \bdv + \reps \nabla_x (\rho_\epsilon + \theta_\epsilon ) & =    n_\epsilon \sdot E_\epsilon +   j_\epsilon \timess B_\epsilon, \\
\partial_t \left(\tfrac{3}{5}\theta_\epsilon - \tfrac{2}{5}\rho_\epsilon \right) + \tfrac{2}{5\epsilon} \divg \intv \hat{B} \mathcal{L}f_\epsilon \bdv   & =  \epsilon \sdot \tfrac{2}{5} j_\epsilon \sdot E_\epsilon.
\end{aligno}
For the intergration term in \eqref{nsp-app},  recalling that
\[ \reps \divg \intv \hat{A} \mathcal{L}f_\epsilon \bdv = \reps \divg \intv \hat{A} \mathcal{L}f_\epsilon^\perp \bdv, \]
Based on \eqref{estwhole},     we can infer that
\begin{align}
\label{estintegrationab}
\int_0^\infty \|\reps \divg \intv \hat{A} \mathcal{L}f_\epsilon(\tau) \bdv\|_{H^{s-1}_x}^2 \bd \tau \le C_0.
\end{align}
With the help of \eqref{estlimitj} and \eqref{estintegrationab}, we can infer that
\begin{align}
\label{estConvergencerhotheta}
 \nabla_x (\rho_\epsilon + \theta_\epsilon  ) \to \rho + \theta= 0, ~~\text{ in the distribution sense. }
\end{align}
Now we can try to deduce the equation of $u$. Let $\mathbf{P}$ be  the Leray projection operator on torus, from   \eqref{fluid-app},  it follows that
\begin{aligno}
\label{nsp-app-4}
\partial_t \mathbf{P} u_\epsilon + \reps \mathbf{P}\left( \divg \intv \hat{A} \mathcal{L}f_\epsilon \bdv \right)   & = \mathbf{P} \left(   n_\epsilon E_\epsilon +   j_\epsilon\timess B_\epsilon \right).
\end{aligno}
Based on the first equation of \eqref{vmbtwo-rewrite}, we can represent $\mathcal{L}(f_\epsilon)$ like this:
\begin{align*}
     \reps \mathcal{L}(f_\epsilon)     & = -\vdot f_\epsilon - \epsilon \partial_t f_\epsilon +  \Gamma(f_\epsilon, f_\epsilon) +   \epsilon  E_\epsilon\sdot v \sdot  h_\epsilon -  ( \epsilon E_\epsilon +   v \times B_\epsilon) \sdot \nabla_v  h_\epsilon    \\
     & =    {\Gamma}(f_\epsilon,f_\epsilon) - \vdot   f_\epsilon  + R_1(\epsilon).
\end{align*}
By simple calculation (see \cite{diogosrm-2019-vmb-fluid,bgl1993convergence,bgl1991formal}),  it follows that
\begin{align}
\label{aaa}
  \intv \hat{A} \sdot \reps \mathcal{L}f_\epsilon \bdv = u_\epsilon \otimes u_\epsilon - \tfrac{|u_\epsilon|^2}{3} \mathbf{I} - \mu \left(\nabla_x u_\epsilon + \nabla^T_\epsilon u_\epsilon -\tfrac{2}{3} \divg u_\epsilon \mathbf{I}\right)- R_f(\epsilon)
\end{align}
with
\[R_f(\epsilon) :=  \intv \hat A \sdot \left( R_1(\epsilon) -  \vdot f^\perp_\epsilon +  {\Gamma}(f_\epsilon^\perp,f_\epsilon) + {\Gamma}(f_\epsilon,f_\epsilon^\perp) \right) \bdv,  \]
and
\[   \mu = \tfrac{1}{15} \sum\limits_{ 1 \le i \le 3 \atop 1 \le j \le 3}\intv A_{ij}\hat{A}_{ij}\bdv. \]
According to \eqref{estwhole},  the microscopic part is $O(\epsilon)$ in $H^s_x$ sense. Thus,    in the distributional sense
\[ R_f(\epsilon) \to 0. \]
Based on \eqref{estwhole}, \eqref{estlimitj}, \eqref{estintegrationab} and \eqref{nsp-app-4},  it follows that
\begin{align}
\partial_t \mathbf{P} \mathbf{u}_\epsilon  \in H^{s-1}_{x}.
\end{align}
By  Aubin-Lions-Simon theorem (see \cite{tools-ns}), we can infer the following strong convergence with time:
\begin{align}
\label{strongu}
\mathbf{P} \mathbf{u}_\epsilon  \in C((0,+\infty;H^{s-1}_{x});~~ \mathbf{P} \mathbf{u}_\epsilon \to    \mathbf{u},~~  \text{in}~~ C((0,+\infty;H^{s-1}_{x}).
\end{align}
Then according to \eqref{estwhole} and \eqref{aaa},   in the distributional sense,
\begin{aligno}
& \mathbf{P} u_\epsilon \to u,~~\reps \mathbf{P}\left( \divg \intv \hat{A} \mathcal{L}h_\epsilon \bdv \right)  \to u \sdot \nabla u - \mu \Delta u.
\end{aligno}
Finally, from  \eqref{estconvergenceMarco}, \eqref{estfluid-j} and \eqref{estlimitn}, we can finally deduce that
\begin{align}
\label{fluidlimitu}
\partial_t u + u\sdot\nabla u - \mu\Delta u + \nabla P =        \curl B \timess B.
\end{align}

{\bf Step 4, the limiting equation of $\theta$.}

By the similar way of deducing \eqref{aaa}, for the approximate temperature equation, we can infer that
\begin{align}
\label{fluid-tem}
\partial_t \left(\tfrac{3}{5}\theta_\epsilon - \tfrac{2}{5}\rho_\epsilon \right)  + \divg(u_\epsilon \theta)- \kappa \Delta \theta_\epsilon    = \epsilon \tfrac{2}{5} j_\epsilon \sdot E_\epsilon + \divg R_\theta(\epsilon),
\end{align}
with
\begin{align}
\label{bbb}
  \tfrac{2}{5}   \intv \hat{B} \sdot \reps \mathcal{L} h_\epsilon \bdv = u_\epsilon\sdot \theta_\epsilon - \kappa \nabla \theta_\epsilon -  R_\theta(\epsilon)
\end{align}
and
\[R_\theta(\epsilon) :=  \tfrac{5}{2}\intv \hat B \sdot \left( R_1(\epsilon) -  \vdot f^\perp_\epsilon +  {\Gamma}(f_\epsilon^\perp,f_\epsilon) + {\Gamma}(f_\epsilon,f_\epsilon^\perp) \right) \bdv,~~\kappa = \tfrac{2}{15} \sum\limits_{ 1 \le i \le 3  }\intv B_{i}\hat{B}_{i}\bdv. \]
By the similar way of deducing \eqref{strongu}, we can infer that
\begin{align}
\label{strongtheta}
\left(\tfrac{3}{5}\theta_\epsilon - \tfrac{2}{5}\rho_\epsilon \right)  \in C((0,+\infty;H^{s-1}_{x});~~ \left(\tfrac{3}{5}\theta_\epsilon - \tfrac{2}{5}\rho_\epsilon \right) \to    \theta,~~  \text{in}~~ C((0,+\infty;H^{s-1}_{x}).
\end{align}
Based on \eqref{estconvergenceMarco} and \eqref{estlimitj}, the right hand of \eqref{fluid-tem} will go to zero in the distributional sense. Finally,  we have
\begin{align}
\partial_t \theta + u \sdot\nabla \theta - \kappa \Delta \theta =0.
\end{align}

{\bf Step 5, the Ohm's law.}

Based on the previous analysis in this section, \eqref{fluid-app} turns to
\begin{align}
\label{nsfp-c}
\begin{cases}
\partial_t u + u\sdot \nabla u - \nu \Delta u + \nabla P = (\curl B) \timess B,\\
\partial_t \theta + u \sdot\nabla \theta - \kappa \Delta \theta =0,\\
\divg u = \divg B= 0,~~ \rho + \theta =0,\\
\partial_t B + \curl E =0.
\end{cases}
\end{align}
We need to represent $E$ in another way. This useful relation hides in the Ohm's law. From \eqref{equationJtilde},
\begin{align*}
  j_\epsilon =  -\epsilon^2 \partial_t \tilde{j}_\epsilon\ + \sigma   E_\epsilon-  \divg  \intv \tilde{v}\otimes v  h_\epsilon  \bdv   + \epsilon   \intv \tilde v  N_1 \bdv .
\end{align*}
For the last two terms in the above equation, recalling the microscopic parts of $f_\epsilon$ and $h_\epsilon$ are $O(\epsilon)$, it follows that
\begin{aligno}
\label{estlimitjj}
 \divg  \intv \tilde{v}\otimes v  h_\epsilon  \bdv & = \divg  \intv \tilde{v}\otimes v  h_\epsilon^\perp  \bdv +\sigma\nabla_x n_\epsilon= \sigma\nabla_x n_\epsilon + O(\epsilon),\\
 \epsilon   \intv \tilde v  N_1 \bdv &  =-   \intv  \left(   v\times B_\epsilon   \sdot  \nabla_v f_\epsilon  \right)  \tilde{v} \bdv + n_\epsilon \intv \Gamma(1,\bdp f_\epsilon) \tilde{v}\bdv + R_4(\epsilon),
\end{aligno}
with
\[ R_4(\epsilon) =  \epsilon\intv   \left( E_\epsilon    \sdot \nabla_v (\m f_\epsilon) \right) \tilde{v} \bd v  +   \intv  \Gamma( h_\epsilon^\perp,f_\epsilon)\tilde{v}\bdv +  \intv  \Gamma( h_\epsilon,f_\epsilon^\perp)\tilde{v}\bdv.  \]
According to  \eqref{estconvergenceMarco},  \eqref{estwhole} and \eqref{estlimitn},   in the distributional sence:
\begin{align}
\divg  \intv \tilde{v}\otimes v  h_\epsilon  \bdv  \to 0,~~\epsilon   \intv \tilde v  N_1 \bdv    \to  \sigma u\timess B.
\end{align}
Based on \eqref{estfluid-j} and \eqref{estlimitj}, in the distributional sense, we have
\begin{align}
\label{estohm}
 j_\epsilon \to j= \curl B=\sigma (    E  +   u\timess B).
\end{align}
Then \eqref{nsfp-c} becomes to MHD system
\begin{align}
\label{mhd}
\begin{cases}
\partial_t u + u\sdot \nabla u - \nu \Delta u + \nabla P = (\curl B) \timess B,\\
\partial_t \theta + u \sdot\nabla \theta - \kappa \Delta \theta =0,\\
\divg u = \divg B= 0,~~ \rho + \theta =0,\\
\partial_t B  - \tfrac{1}{\sigma} \Delta B = \curl(u \times B) .
\end{cases}
\end{align}

\appendix

\section{Formal derivation}
\label{sec-formal}


\begin{thebibliography}{10}

\bibitem{twonstomhd}
D.~Ars\'{e}nio, S.~Ibrahim and N.~Masmoudi,
\newblock A derivation of the magnetohydrodynamic system from
  {N}avier-{S}tokes-{M}axwell systems.
\newblock \emph{Arch. Ration. Mech. Anal.} \textbf{216} (2015), 767--812.

\bibitem{diogosrm-2019-vmb-fluid}
D.~Ars\'{e}nio and L.~Saint-Raymond.
\newblock \emph{From the {V}lasov-{M}axwell-{B}oltzmann system to
  incompressible viscous electro-magneto-hydrodynamics. {V}ol. 1}.
\newblock EMS Monographs in Mathematics. European Mathematical Society (EMS),
  Z\"{u}rich (2019).

\bibitem{bgl1991formal}
C.~Bardos, F.~Golse and C.~D. Levermore,
\newblock Fluid dynamic limits of kinetic equations. {I}. {F}ormal derivations.
\newblock \emph{J. Statist. Phys.} \textbf{63} (1991), 323--344.

\bibitem{bgl1993convergence}
C.~Bardos, F.~Golse and C.~D. Levermore,
\newblock Fluid dynamic limits of kinetic equations. {II}. {C}onvergence proofs
  for the {B}oltzmann equation.
\newblock \emph{Comm. Pure Appl. Math.} \textbf{46} (1993), 667--753.

\bibitem{bu1991}
C.~BARDOS and S.~UKAI,
\newblock The classical incompressible navier-stokes limit of the boltzmann
  equation.
\newblock \emph{Mathematical Models and Methods in Applied Sciences}
  \textbf{01} (1991), 235--257.

\bibitem{tools-ns}
F.~Boyer and P.~Fabrie.
\newblock \emph{Mathematical tools for the study of the incompressible
  {N}avier-{S}tokes equations and related models}, \emph{Applied Mathematical
  Sciences}, vol. 183.
\newblock Springer, New York (2013).

\bibitem{briant-2015-be-to-ns}
M.~Briant,
\newblock From the {B}oltzmann equation to the incompressible {N}avier-{S}tokes
  equations on the torus: a quantitative error estimate.
\newblock \emph{J. Differential Equations} \textbf{259} (2015), 6072--6141.

\bibitem{diperna-lions1989cauchy}
R.~J. DiPerna and P.-L. Lions,
\newblock On the {C}auchy problem for {B}oltzmann equations: global existence
  and weak stability.
\newblock \emph{Ann. of Math. (2)} \textbf{130} (1989), 321--366.

\bibitem{dlyz2017cmp}
R.~Duan, Y.~Lei, T.~Yang and H.~Zhao,
\newblock The {V}lasov-{M}axwell-{B}oltzmann system near {M}axwellians in the
  whole space with very soft potentials.
\newblock \emph{Comm. Math. Phys.} \textbf{351} (2017), 95--153.

\bibitem{gsrm2004}
F.~Golse and L.~Saint-Raymond,
\newblock The {N}avier-{S}tokes limit of the {B}oltzmann equation for bounded
  collision kernels.
\newblock \emph{Invent. Math.} \textbf{155} (2004), 81--161.

\bibitem{uvpb2020}
M.~Guo, N.~Jiang and Y.-L. Luo,
\newblock From vlasov-poisson-boltzmann system to incompressible
  navier-stokes-fourier-poisson system: convergence for classical solutions.
\newblock \emph{arXiv preprint arXiv:2006.16514}  (2020).

\bibitem{guo-2003-vmb-invention}
Y.~Guo,
\newblock The {V}lasov-{M}axwell-{B}oltzmann system near {M}axwellians.
\newblock \emph{Invent. Math.} \textbf{153} (2003), 593--630.

\bibitem{guo2006NSlimit}
Y.~Guo,
\newblock Boltzmann diffusive limit beyond the {N}avier-{S}tokes approximation.
\newblock \emph{Comm. Pure Appl. Math.} \textbf{59} (2006), 626--687.

\bibitem{vmbtonsp}
J.~Jang,
\newblock Vlasov-{M}axwell-{B}oltzmann diffusive limit.
\newblock \emph{Arch. Ration. Mech. Anal.} \textbf{194} (2009), 531--584.

\bibitem{jama2012siam}
J.~Jang and N.~Masmoudi,
\newblock Derivation of {O}hm's law from the kinetic equations.
\newblock \emph{SIAM J. Math. Anal.} \textbf{44} (2012), 3649--3669.

\bibitem{vmbtonswu}
N.~Jiang and Y.-L. Luo,
\newblock From {V}lasov-{M}axwell-{B}oltzmann system to two-fluid
  incompressible {N}avier-{S}tokes-{F}ourier-{M}axwell system with ohm's law:
  convergence for classical solutions.
\newblock \emph{arXiv preprint arXiv:1905.04739}  (2019).

\bibitem{vmbtonswh}
N.~Jiang, Y.-L. Luo and T.-F. Zhang,
\newblock Incompressible navier-stokes-fourier-maxwell system with ohm's law
  limit from vlasov-maxwell-boltzmann system: Hilbert expansion approach.
\newblock \emph{arXiv preprint arXiv:2007.02286}  (2020).

\bibitem{ns-limit-2018}
N.~Jiang, C.-J. Xu and H.~Zhao,
\newblock Incompressible {N}avier-{S}tokes-{F}ourier limit from the {B}oltzmann
  equation: classical solutions.
\newblock \emph{Indiana Univ. Math. J.} \textbf{67} (2018), 1817--1855.

\bibitem{jz2020vpbconvergence}
N.~Jiang and X.~Zhang,
\newblock Sensitivity analysis and incompressible {N}avier-{S}tokes-{P}oisson
  limit of {V}lasov-{P}oisson-{B}oltzmann equations with uncertainty.
\newblock \emph{arXiv preprint arXiv:2007.00879}  (2020).

\bibitem{lm2010soft}
C.~D. Levermore and N.~Masmoudi,
\newblock From the {B}oltzmann equation to an incompressible
  {N}avier-{S}tokes-{F}ourier system.
\newblock \emph{Arch. Ration. Mech. Anal.} \textbf{196} (2010), 753--809.

\bibitem{vpb2020limit-spectrum}
H.-L. Li, T.~Yang and M.~Zhong,
\newblock Diffusion limit of the vlasov-poisson-boltzmann system.
\newblock \emph{arXiv preprint arXiv:2007.01461}  (2020).

\bibitem{lm2001acoustic}
P.-L. Lions and N.~Masmoudi,
\newblock From the {B}oltzmann equations to the equations of incompressible
  fluid mechanics. {I}, {II}.
\newblock \emph{Arch. Ration. Mech. Anal.} \textbf{158} (2001), 173--193,
  195--211.

\bibitem{masmoudi-srm2003stokesfourier}
N.~Masmoudi and L.~Saint-Raymond,
\newblock From the {B}oltzmann equation to the {S}tokes-{F}ourier system in a
  bounded domain.
\newblock \emph{Comm. Pure Appl. Math.} \textbf{56} (2003), 1263--1293.

\bibitem{mischler2010asens}
S.~Mischler,
\newblock Kinetic equations with {M}axwell boundary conditions.
\newblock \emph{Ann. Sci. \'Ec. Norm. Sup\'er. (4)} \textbf{43} (2010),
  719--760.

\bibitem{mouhot-2006-homogeneous}
C.~Mouhot,
\newblock Rate of convergence to equilibrium for the spatially homogeneous
  {B}oltzmann equation with hard potentials.
\newblock \emph{Comm. Math. Phys.} \textbf{261} (2006), 629--672.

\bibitem{mouhotneumann-2006-decay}
C.~Mouhot and L.~Neumann,
\newblock Quantitative perturbative study of convergence to equilibrium for
  collisional kinetic models in the torus.
\newblock \emph{Nonlinearity} \textbf{19} (2006), 969--998.

\bibitem{saint2009book}
L.~Saint-Raymond.
\newblock \emph{Hydrodynamic limits of the {B}oltzmann equation}, \emph{Lecture
  Notes in Mathematics}, vol. 1971.
\newblock Springer-Verlag, Berlin (2009).

\bibitem{sr2006-vmb}
R.~M. Strain,
\newblock The {V}lasov-{M}axwell-{B}oltzmann system in the whole space.
\newblock \emph{Comm. Math. Phys.} \textbf{268} (2006), 543--567.

\bibitem{twovpblimits}
Y.~Wang,
\newblock The diffusive limit of the vlasov–boltzmann system for binary
  fluids.
\newblock \emph{SIAM Journal on Mathematical Analysis} \textbf{43} (2011),
  253--301.

\bibitem{mvmbtothree}
X.~Zhang,
\newblock The diffusive limits of two species {V}lasov-{M}axwell-{B}oltzmann
  equations.
\newblock \emph{arXiv preprint arXiv:2103.16881}  (2021).

\end{thebibliography}
\end{document}